\documentclass[11pt]{amsart}
\usepackage{amssymb,amsmath,amsthm,mathtools}
\usepackage{a4wide}
\usepackage{graphicx}
\usepackage[utf8]{inputenc} 
\usepackage{microtype}
\usepackage{hyperref}
\usepackage{amsfonts} 
\usepackage{latexsym}
\usepackage[font=small,format=hang,labelfont={sf,bf}]{caption}
\usepackage{epsfig}
\usepackage{subfig}
\usepackage{url}
\usepackage{varioref}
\usepackage{bm}
\usepackage{tikz}
\usepgflibrary{patterns} 
\usepgflibrary[patterns] 
\usetikzlibrary{patterns} 
\usetikzlibrary[patterns] 
\usepackage{amsfonts} 
\usepackage{geometry}
\usepackage{color}
\usepackage{xcolor}
\usepackage{mathrsfs}

\usepackage{latexsym}
\usepackage{bbm}
\usepackage{dsfont}
\numberwithin{equation}{section}
\theoremstyle{plain}
\begingroup
\newtheorem{theorem}{Theorem}[section]
\newtheorem{lemma}[theorem]{Lemma}
\newtheorem{proposition}[theorem]{Proposition}
\newtheorem{corollary}[theorem]{Corollary}
\endgroup
\def\beq {\begin{equation}}
\def \eeq {\end{equation}}
\theoremstyle{definition}
\begingroup

\endgroup

\def\F{\mathcal F}

\definecolor{ddorange}{rgb}{1,0.5,0}
\definecolor{ddcyan}{rgb}{0,0.2,1.0}

\newcommand{\e}{\varepsilon}
\newcommand{\Ha}{\mathcal{H}}

\newcommand{\numberset}{\mathbb}
\newcommand{\R}{\numberset{R}}

\newcommand{\N}{\numberset{N}}

\newcommand{\pa}{\partial}

\newcommand{\E}{\mathcal{E}}

\title[Fractional heat equation and Perimeters]{Characterization of sets of finite local and non local perimeter via non local heat equation}
\author[A. Kubin]
{A. kubin}
\address[Andrea Kubin]{
	Jyv\"askyl\"an Yliopisto, Matematiikan ja Tilastotieteen Laitos, Jyv\"askyl\"a, Finland
}
\email[A. Kubin]{andrea.a.kubin@jyu.fi}

\author[D.A. La Manna]
{Domenico Angelo  La Manna}
\address[Domenico Angelo La Manna]{
	Università degli Studi di Napoli Federico II, Dipartimento di Matematica ed Applicazioni R. Caccioppoli.
}
\email[D. A. La Manna]{domenicoangelo.lamanna@unina.it}

\begin{document}

	\begin{abstract}	
	In this paper we provide a characterization of sets of finite local and non local perimeter via a $\Gamma-$convergence result. As an application we give a short proof of the isoperimetric inequality, both in the local and in the non local case.
\end{abstract}

\maketitle


	\tableofcontents

	\section*{Introduction}
	In this article we analyze the asymptotic behaviour of the energy
	 \begin{equation}\label{energia}
	  \mathcal{E}_t^s(E):= \int_{E} \int_{E^c} P^s(x-y,t) dy dx 
	\end{equation}
as $ t \rightarrow 0^+$  for all $ s \in (0,1)$ and
 where $P^s(z,t)$ is the fundamental solution of the fractional heat equation.
In the spirit of \cite{BBM,MPPP,GP,CM}, our goal is to provide a characterization of sets of finite local and nonlocal perimeter via the fractional heat semigroup. In \cite{BBM} they prove a characterization of the Sobolev space using as energy the $s$-Gagliardo seminorms, they prove also pointwise convergence result as the parameter $s \rightarrow 1$ (see also \cite{CDLKNP,KP,KT} for a $\Gamma$-convergence result). Recently is appear \cite{GennaioliStefani}, where the authors study a similar problem. The starting point of the investigation in \cite{MPPP} is the generalization of the asymptotic of the heat semigroup shown in \cite{L} to prove the isoperimetric inequality via the heat semigroup. Their theorem provides a characterization of sets of finite perimeter via the heat semigroup.  
 Since the kernel $\|P^s(\cdot,t)\|_{ \L^1(\R^n)}=1$ for all $t>0$ it is quite immediate to understand that the functional in \eqref{energia} is well-defined 
for all measurable sets $E$ such that $E$ or $E^c$ has finite measure. A simple computation shows that when $t \to 0$ then $\mathcal  E _t^s(E)\to 0$ (see the prof of Theorem \ref{mainthm1}) whenever $E$ is a set of finite measure.
Our scope in this article is to find the first non trivial order of the power series expansion of such a functional under suitable assumptions. 

Let us introduce the function 
\beq\label{eq:definizioneg}
g_s(t)=
\begin{cases}
t \qquad & \text{if } s\in (0,\frac12)
\\
t \log t  \qquad &\text{if } s=\frac12
\\
t^\frac{1}{2s} &\text{if }s\in (\frac12,1)  .
\end{cases}
\eeq
This function describes the leading order of $\mathcal E_{t}^s$ as $t$ approaches zero. Let us just synthetically state our (meta)theorem
as follows.
\begin{theorem}
Let $n\geq 2$ and $s\in (0,1)$
and define $\mathbb X^s= \rm{BV}(\R^n)$ if $s\geq \frac12$ and $\mathbb X^s=H^{s}(\R^n)   $ if $s<\frac12$.
Then
\[
\Gamma-\lim_{t\to 0}\frac{\mathcal E_s^t(E)}{g_s(t)}=
\begin{cases}
 \Gamma^{n,s}{\rm P}_{\min\{2s,1\}}(E)      \quad &\text{if  }\; \chi_E\in \mathbb X^s
\\
+\infty &\text{otherwise.}
\end{cases}
\]
\end{theorem}
In the above theorem ${\rm P}_{2s}(E)$ stands for the fractional perimeter, defined in \ref{fractperim}
and ${\rm P}_1 (E)$ is the De Giorgi perimeter of the set $E$ defined in \eqref{DeGiorgiP}.
Let us comment on the reasons to explore the topic under investigation. 
 Laux and Otto in \cite{LO} proved the convergence of the thresholding scheme towards a variational solution of the mean curvature flow under suitable assumptions. It seems natural to extend such a result when one consider the nonlocal heat equation.
This $\Gamma-$convergence result is a first step towards the understanding of the nonlocal thresholding scheme from a variational point of view.
 We mention that in \cite{CS}  Caffarelli and Souganidis, studied the convergence of nonlocal threshold dynamics, showing that they approximate viscosity solutions to the mean curvature and non local mean curvature flow for $s\in [\frac12,1)$ and $s\in (0,\frac12)$ respectively.
\\
One of the first difficulty in proving this theorem comes from the fact that the first non trivial order in the expansion change with respect to the parameter $s$ (see the definition of the function $g_s(t)$ in \eqref{eq:definizioneg}) and therefore the first thing to do is to find the correct space-time scaling which ensures some basic properties, such as compactness.
Once the function $g_s$ is detected, we prove the compactness Theorem \ref{Compthm}. When $s\not =\frac12$ the proof relies on first showing compactness in $\mathrm{L}^1$ via the Frechet-Kolmogorov criterion and then we show that the limit belongs to the correct class.
When $s=\frac12$ the logaritmic behavior makes things slightly different and we prove our result in a completely different way Inspired by \cite{DLKP}. 
After good compactness properties are established, we can prove our theorem and to do so we distinguish between two cases as the proof is significantly different: in case $s<\frac12$ the proof does not need any geometric measure theory tool, as we just need to carefully use Fourier analysis. 
The main reason for this fact is that  the fractional laplacian of a characteristic function, for $s<\frac12$, is not only merely a measure but it is a function. There fore, it is reasonable to think that a strategy based on Fourier analysis provides a fast track to the goal and it is actually the case.\\
For $s\geq \frac12$ things change completely as the functional  starts to lose its nonlocal nature and there fore we need to borrow some tools from geometric measure theory. The strategy we decide to follow is the well established one used in \cite{ADPM} to study the $\Gamma-$convergence of the fractional perimeter as the differentiation parameter $s$ goes to one (see also \cite{CCLP}, \cite{DL} for the Gaussian case). In our case,
another difficulty comes from the fact that the fractional heat kernel is not known when $s>\frac12$ and to deal with it we once again need to use Fourier analysis. We stress that even though some decay properties are known (see \eqref{eq:decay}), such a knowledge is not enough to prove $\Gamma-$convergence, as the sole use of such estimates would provide different constant  in the $\Gamma-\liminf$ and $\Gamma-\limsup$.
Note that the case $s>\frac12$ could be recovered from \cite{AB}, but in our case we are able to provide the explicit form of the limit functional. For this reason we believe that it is worth to highlight it in here. In fact, in Section \ref{sec:gammalim} we will provide the argument to find the optimal constant by carefully computing the functional on a halfspace and then we procede by proving our convergence result. We do not enter into details of the proof of the equality $\Gamma^{n,s}=\Gamma_{n,s}$  (see \eqref{eq:defgamma} and \eqref{eq:liminf2}  for the definition of these two contants) since the proof follows without any change by a standard argument (gluing lemma and calibration) as in \cite{BP}
\\ 
To conclude, we use our result to prove the isoperimetric inequality, both in the local and nonlocal framework. The proof relies on Hardy rearrangement inequality (the approximating kernels are $\mathrm{L}^1$ normalized-functions is the space variable)
and our $\Gamma-$convergence result.
Before concluding this introduction, let us comment our result with a little heuristic interpretation. By Bochner subordination formula, one way to obtain the fractional laplacian of a function is
\[
(-\Delta)^s u= \int_0^\infty \frac{e^{t\Delta}u -u}{t^{1+s}}\,dt
\]
and therefore, if $s<\frac12$ and $u=\chi_E$
\[
{\rm P}_{2s}(E)= \int_{\R^n}u\Delta^s u\, dx= \int_0^\infty \frac{1}{t^{1+s}} \int_{\R^n} u(x)(e^{t\Delta}u(x) -u(x))
\,dx\,dt.
\]
Hence, the fractional perimeter of a set $E$ can be computed by taking the evolution of the characteristic function of $E$ via the heat equation and then computing the $s-$ derivative at time $0$ of the energy of the evolving solution at time $t/2$.
One of the results of this paper can be formally read as
\[
\lim_{t\to 0^+} \frac{1}{t} \int_{\R^n} u(x)(e^{-t^{2s}(-\Delta)^s}u(x) -u(x))= 
\Gamma_{n,s} {\rm P}(E)
\,dx
\]
for $s>\frac12$. In this sense, our result resembles the famous Bochner subordination formula for the perimeter of a set: we found that the perimeter of a set $E$ can be recovered by considering first the evolution of the characteristic function of $E$ via the fractional heat equation evaluated at time $t=t^{2s}$ and secondly computing the limit above, which is essentially the derivative at time $0$ of the energy of the evolving solution at time $\frac12 t^{2s}$. So, in a sense, we are saying that given a measurable set $E$ smooth enough, if we first compute the evolution via {\it local} heat equation and then take a {\it non local} time derivate of an appropriate energy we recover the {\it non local} perimeter, while if we first compute the evolution of $\chi_E$ via {\it non local}  heat equation and after we take the {\it local} time derivative of an appropriate energy, we recover the {\it local} perimeter.
\section{Notation}
We recall some notation and basic results of geometric measure theory from \cite{AFP}. 
We denote with $\mathrm{M}(\R^n)$ the set of all Lebesgue measurable subset of $\R^n$. For all $ k \in \N$ we denote with $\omega_k$ the Lebesgue measure of the unit ball of $\R^k$. For every $ E \in \mathrm{M}(\R^n)$ we denote by ${\rm P}(E)$ the De Giorgi perimeter of $E$ defined by
\begin{equation}\label{DeGiorgiP}
	{\rm P}(E)= \sup \left\{ \int_{E} \mathrm{div} \phi (x)dx \colon \, \phi \in C_c^1(\R^n,\R^n) \text{ and } \| \phi \|_{\infty} \leq 1 \right\}.
\end{equation}
For every $ E \in \mathrm{M}(\R^n)$  the set $ \partial^* E$ identifies the reduced boundary of $E$ and $\nu_{E} : \partial^* E \rightarrow \R^n$ the outer normal vector field. 
For a function $u\in \mathrm{L}^2(\R^n)$ we denote the Fourier transform of $u$ by $\F [u(\cdot)](\xi)$, i.e.
\[
\F [u(\cdot)](\xi)=\frac{1}{(2\pi)^\frac{n}{2}}\int_{\R^n} u(x)e^{-i\langle x,\xi \rangle}\, dx
\]
and note that the Fourier transform is invertible and
\[
\F^{-1} [v(\cdot)](x)=\frac{1}{(2\pi)^\frac{n}{2}}\int_{\R^n} v(\xi)e^{-i\langle x,\xi \rangle}\, d\xi
\]
For $s\in (0,1)$ 
and a smooth function $u:\R^n\to \R$
we consider the operator
\[
(-\Delta)^s u(x)= C_{n,s} \int_{\R^n} \frac{u(x)-u(y)}{|x-y|^{n+2s}}\, dy
\]
where \begin{equation}\label{costante!}
	C_{n,s}= \left(\int_{\R^n} \frac{1-\cos(h_1)}{\vert h\vert^{n+2s}} dh\right)^{-1}.
\end{equation} 
For all $ s \in (0,1)$ we define the $s$-fractional Gagliardo seminorm of a function $ u \colon \R^n \rightarrow \R$ as
\begin{equation}
	[u]_{H^s(\R^n)}:= \left(\int_{\R^n} \int_{\R^n} \frac{ \vert u(x)-u(y)\vert^2}{\vert x-y \vert^{n+2s}} dx dy\right)^{\frac{1}{2}}
\end{equation}
and we denote with $ \mathrm{H}^s(\R^n)$ the space of the function $ u \in \mathrm{L}^2(\R^n)$ such that $[u]_{H^s(\R^n)}< +\infty$.
The fractional Gagliardo seminorm can be written in Fourier as
\begin{equation}
	[u]_{H^s(\R^n)}^2= 2 C_{n,s}^{-1} \int_{\R^n} \vert \xi \vert^{2s} \vert \F[u](\xi) \vert^2 d \xi.
\end{equation}
If $ s \in (0,\frac{1}{2})$ and $ E \in \mathrm{M}(\R^n)$ we define the $2s$-fractional perimeter as
\begin{equation}\label{fractperim}
	{\rm P}_{2s}(E):= \int_{E} \int_{E^c} \frac{1}{\vert x-y \vert^{n+2s}}dx dy
\end{equation}
and we observe that $ \mathrm{P}_{2s}(E)= \frac{1}{2} [\chi_{E}]_{H^s(\R^n)}$.
We also recall that for all $ u : \R^n \rightarrow \R$ smooth
\[
\F  [(-\Delta)^s u(\cdot)](\xi)= |\xi|^{2s} \F [u(\cdot)] (\xi).
\]
Given $u_0\in \mathrm{L}^\infty(\R^n)$ let us consider the solution to fractional heat equation
\beq\label{eq:cauchyproblem}
\begin{cases}
\pa_t v+ (-\Delta)^s v= 0\\
v(x,0)=u_0(x).
\end{cases}
\eeq
We now recall some basic results about the fundamental solution of the heat equation (see \cite{G} for a detailed discussion and proofs). 
We define the fundamental solution of the fraction heat equation the function $P^s(z,t)$ is such that 
\begin{equation}\label{fondsolu}
	\begin{cases}
		\pa_t  P^s(z,t) + (-\Delta)^s  P^s(z,t) = 0 \quad t>0, \, z \in \R^n,\\
		P^s(\cdot,0)=\delta_0
	\end{cases}
	\quad
	\int_{\R^n}P^s(z, t)dz=1
\end{equation}
where $ \delta_0$ is the Dirac delta measure in $\R^n$.
It is well known that for $u_0\in \mathrm{H}^s(\R^n)\cap \mathrm{L}^\infty(\R^n)$ the solution of the problem \eqref{eq:cauchyproblem} can be written as a convolution between the fundamental solution and the initial data, i.e.
\begin{equation}\label{solprbCauchy}
	v(x,t)=\int_{\R^n} P^s(x-y,t) u_0(y)\, dy
\end{equation}
and $ v \in C([0,+\infty), \mathrm{H}^s(\R^n)) \cap C^1([0,+\infty), \mathrm{L}^2(\R^n)) \cap C^{\infty}((0,+\infty), C^{\infty}(\R^n))$. 
Moreover, the function $P^s(z,t)$ satisfies
\beq\label{eq:fourier}
\F [P^s(\cdot, t)] (\xi)= \frac{1}{(2\pi)^{\frac{n}{2}}} e^{-t|\xi|^{2s}}
\eeq
for $s\in (0,1)$.
The function $P^s$ can not be explicitly computed unless in the particular case $s=\frac12$, for which we know
\begin{equation}\label{eng12}
P^\frac12(z,t)= \frac{t}{(|z|^2+t^2)^{\frac{n+1}{2}}}.
\end{equation}
When $s\not= \frac12$ 
it is well known that $P^s(z,t)$ satisfies the decay estimates
\beq \label{eq:decay}
c(n,s) \left( \frac{t}{|z|^{n+2s}} \wedge  t^{-\frac{n}{2s}} \right) \leq P^s(z,t)  \leq C(n,s) \left( \frac{t}{|z|^{n+2s}} \wedge  t^{-\frac{n}{2s}} \right)
\eeq
and the scaling property
\beq\label{eq:scaling}
P^s(z,t)= t^{-\frac{n}{2s}}P^s(z t^{-\frac{1}{2s}},1).
\eeq
\section{Compactness}
In this section we study the compactness properties of the sequence $\{E_i\}_{i \in \N}$ such that $ E_i \subset U$ for all $i \in \N$, with $U$ be open and bounded, and $\frac{1}{g_s(t_i)} \mathcal{E}^s_{t_i}(E_i) \leq M$ for all $i \in \N$, with $M\in \R$, and $ t_i \rightarrow 0^+$ as $i \rightarrow +\infty$, where
\begin{equation}\label{energias<1/2}
	\frac{1}{g_s(t)} \mathcal{E}^s_{t}(E)= \frac{1}{g_s(t)} \int_{E} \int_{E^c} P^s(x-y,t)dx dy \text{ for } s \in \big(0,1\big)
\end{equation}
 and $  P^s(x,t)$ is the solution of \eqref{fondsolu}
and $g_s(t)$ is the function defined in\eqref{eq:definizioneg}.
This compactness properties suggests the candidate $\Gamma$-limit: for $s\in (0,\frac12)$ such a candidate is the $2s$-fractional perimeter defined in \eqref{fractperim} while for $s\geq\frac{1}{2}$ is the classical perimeter \eqref{DeGiorgiP}. The main theorem of this section is the following.
\begin{theorem}\label{Compthm}
	Let $E_{t_i}\subset U$ for  all $ i \in \N$ and $t_i \rightarrow 0^+$ where $U \subset \R^n$ open and bounded. Assume that 
	\beq\label{eq:compacAAA}
	\frac{\mathcal {E}_{t_i}^s (\chi_{E_{t_i}})}{g_s(t_i)}\leq M. 
	\eeq
	There exists $ E \subset U$ such that, up to a subsequence $\chi_{E_{t_i}}\to \chi_E$ in $ \mathrm{L}^1(\R^n)$.
	Moreover, if $s\in (0,\frac12)$ we have
	${\rm P}_{2s} (E) <\infty$, while for $s\in [\frac12,1)$ it holds ${\rm P}(E)<\infty$.
\end{theorem}
We note that for $s\not =\frac12$ our proof strongly relies on properties of solutions to the fractional heat equation, while for $s=\frac12$ a more delicate geometric argument is needed to catch the precise behavior of $\mathcal {E}_{t_i}(E_i)$ as the sequence $t_i\to 0$. 

\subsection{Proof of Theorem \ref{Compthm} when $ s \in (0,\frac{1}{2}) \cup (\frac{1}{2},1)$ }
 When $2s<1$, as it will also be clear later, the Fourier transform plays a key role.  
\begin{proof}[Compactness for $2s<1$.]
Let $E_i\subset \R^n$ and $t_i$ as in \eqref{eq:compacAAA} and set $u_i=\chi_{E_i}$, $u_i(x,t)= P^s(\cdot, t) \ast \chi_{E_i}$,  and $v_i(x)= u_i(x,t_i)$. 
We start observing two basic properties of solutions to \eqref{eq:cauchyproblem}:
the functions
\begin{equation}\label{eq:decreasing}
t \in (0, \infty) \to \|u(x,t)\|_{L^2(\R^n)}  \text{ and } t \in (0,\infty)\to  \|u(x,t)\|_{H^s(\R^n)} 
\end{equation}
are not increasing. The proof of this fact is quite immediate, and we skip it.
Observe that
\[
\begin{split}
\mathcal {E}_{t_i}^s (\chi_{E_{t_i}})=& \int_{\R^n} (1-\chi_{E_i}) u_i(x,t_i)\, dx =\int_{\R^n} u_i(x,0)^2- u_i(x,t_i/2)^2 \,dx
\\
=&-\int_{0}^{\frac{t_i}{2}} \frac{d}{d\tau }\int_{\R^n}u_i^2(x,\tau)\,dx\, d\tau
=-2\int_{0}^{\frac{t_i}{2}} \int_{\R^n}u_i(x,\tau)\pa_{\tau} u_i(x,\tau)\,dx\, d\tau
\\
=&2 \int_{0}^{\frac{t_i}{2}} \|u_i(\cdot, \tau ) \|_{\mathrm{H}^{s}(\R^n)}^2\, d\tau
\end{split}
\]
where in the last equality we have used $ \| f \|_{H^s(\R^n)}^2= \int_{\R^n} f(x) (-\Delta)^s f(x)dx$.
Therefore, using \eqref{eq:decreasing} we obtain 
\[
 \|v_i\|_{H^s(\R^n)}\leq 
\frac{\mathcal {E}_{t_i}^s (\chi_{E_{t_i}})}{t_i}\leq M.
\]
Since the sequence $v_i$ is bounded in $H^s$, it contains a subsequence converging in $\mathrm L^2(\R^n)$ and in $H^s-\mathrm{weak}$ . With an abuse of notations, we still denote such a subsequence as $v_i$.
It is immediate to show that also $u_i \to v$ in $\mathrm L^2$ and threfore, still up to a subsequence, $\chi_{E_i} \to v$ almost everywhere, hence $v$ takes only the values $0$ and $1$ and is the characteristic function of a set $E$. To conclude,  we observe that $\mathrm{P}_{2s}(E)<\infty$ is a simple consequence  of the semicontinuity of $ u \rightarrow \| u \|_{H^s(\R^n)}$ respect the weak topology.
\end{proof}
\begin{proof}[Compactness for $2s>1$]
In this case we follows from the estimate
\[
|\nabla P^s(x-y,t) |\leq \frac{C}{t^{\frac{1}{2s}}} P^s(x-y,t).
\]

In fact, setting again $v_i(x)= P^s(\cdot,t_i)\ast \chi_{E_i}$, we have that $v_i$ is smooth and 
\[\begin{split}
\int_{\R^n}|\nabla v_i|\, =\int_{\R^n} | \nabla& P^s(\cdot,t_i)\ast \chi_{E_i}\, dy|\, dx
=
\int_{\R^n} \left| \int_{\R^n} \nabla P^s (x-y,,t_i) \chi_{E_i}(y)\, dy\right|\, dx
\\
=&  \int_{\R^n} \left| \int_{\R^n} \nabla P^s (x-y,,t_i)( \chi_{E_i}(y) -\chi_{E_i}(x))\, dy\right|\, dx
\\
\leq &\frac14  \int_{\R^n} \int_{\R^n} \left| \nabla P^s (x-y,,t_i) \right| ( \chi_{E_i}(y) -\chi_{E_i}(x)))^2\, dx\, dy
\\
\leq&  \frac{C_n}{t_i^{\frac{1}{2s}}} \int_{E^c_i}\int_{E_i}  P^s(x-y,t_i) \, dx \leq C_1,
\end{split}
\]
where in the last inequality we used \eqref{eq:compacAAA}. Thus, there exists  $v\in \mathrm L^1(\R^n)$ such that, up to a subsequence, $v_i \to v$ in $ \mathrm L^1_{\text{loc}}$, and in particular  $v_i\to v$. Observe that it is easy to show that $\chi_{E_i}\to v$ in $\mathrm L^1$ (this is because $E_i\subset U$ and $U$ as assumed to be bounded ) as well and, arguing as in the case $s<\frac12$ we can extract a further subsequence to show that $v= \chi_E$.
Observe that, by semicontinuity of the total variation, it is immediate that
 $v=\chi_E$ is a $\mathrm{BV}(\R^n)$ function, which means that $E$ has finite perimeter. 
\end{proof}
\subsection{Proof of Theorem \ref{Compthm} when $ s = \frac{1}{2}$ }
This subsection is devoted to the proof of Theorem \ref{Compthm} when $ s \geq \frac{1}{2}$. Our proof it is inspired
to the compactness result in \cite{DLKP}.
To accomplish this task
we will need some preliminary results. We first recall the following classical result
\begin{theorem}[Compactness in $ \mathrm{BV}$]\label{BVcompatness}
	Let $ \Omega \subset \R^n$ be an open set and let $ \{u_i\}_{i \in \N} \subset \mathrm{BV}_{loc}(\Omega)$ with
	\begin{equation}
		\sup_{i \in \N} \left\{ \int_{A} \vert u_i(x) \vert dx + \vert D u_i \vert(A)\right\}< +\infty \quad \forall A \subset \! \subset \Omega.
	\end{equation}
Then, there exist a subsequence $\{i_k\}_{k \in \N}$ and a function $ u \in \mathrm{BV}_{loc}(\Omega)$ such that $ u_{i_k} \rightarrow u$ in $\mathrm{L}^1_{loc}(\Omega)$ as $ k \rightarrow + \infty$.
\end{theorem}
Now we prove a non-local Poincar\'e-Wirtinger type inequality. In what follows we set $ Q:= [0,1)^n$.
\begin{lemma}
	Let $\xi \in \R^n$ and $ u \in \mathrm{L}^1(lQ+\xi)$ where $0<t<l $. Then we have 
	\begin{equation}\label{05062024pom1}
		\int_{lQ+\xi} \left|u(y)- \frac{1}{l^n} \int_{lQ+\xi} u(x)dx \right|dy \leq l C(n) \int_{lQ+\xi} \int_{lQ+\xi}  \frac{\vert u(x)-u(y) \vert }{(\vert x-y \vert^2+t^2)^{\frac{n+1}{2}}}dx dy.
	\end{equation}
\end{lemma}
\begin{proof}
	By translational invariance, it is enough to prove the claim only for $\xi=0$. By assumption we have
	\begin{equation}
		\begin{split}
			\int_{lQ} \left|u(y)- \frac{1}{l^n} \int_{lQ} u(x)dx \right|dy &\leq \frac{1}{l^n}\int_{lQ} \int_{lQ} \left|  u(x)-u(y)\right| dx dy \\
			&=  \frac{1}{l^n}\int_{lQ} \int_{lQ}   \frac{\vert u(x)-u(y) \vert }{(\vert x-y \vert^2+t^2)^{\frac{n+1}{2}}} (\vert x-y \vert^2+t^2)^{\frac{n+1}{2}} dx dy \\
			&\leq l C(n) \int_{lQ} \int_{lQ}   \frac{\vert u(x)-u(y) \vert }{(\vert x-y \vert^2+t^2)^{\frac{n+1}{2}}}dx dy
		\end{split}
	\end{equation}
i.e., \eqref{05062024pom1}.
\end{proof}
\begin{lemma}
	Let $0 < t <l $. For every $ \xi \in \R^n$ and for every $E \in \mathrm{M}(\R^n)$, it holds
\begin{equation}\label{ambarabacci}
	\frac{1}{l^n} \vert (lQ+\xi) \setminus E \vert \vert (lQ+\xi) \cap E \vert 
 \leq \int_{lQ+\xi} \big|\chi_E(x)- \frac{1}{l^n} \int_{lQ+\xi} \chi_{E}(y)dy\big|dx.
\end{equation}
\end{lemma}
\begin{proof}
	We can assume without loss of generality that $\xi =0$. Then we have that
	\begin{equation}
		\begin{split}
			&\int_{lQ} \big|\chi_{E}(x)-\frac{1}{l^n} \int_{lQ} \chi_{E}(y)dy\big|dx =\int_{lQ \cap E} \big|1-\frac{\vert lQ \cap E \vert}{l^n}\big|dx+ \int_{lQ \setminus E } \frac{\vert lQ \cap E \vert}{l^n} dx\\
			& = \frac{1}{l^n}\big(\int_{lQ \cap E} \vert lQ \setminus E \vert dx+ \int_{lQ \setminus E} \vert lQ \cap E \vert dx\big)= \frac{2}{l^n} \int_{lQ \setminus E} \vert lQ \cap E \vert dx \geq \frac{\vert lQ \cap E \vert \vert lQ \setminus E \vert}{l^n}.
		\end{split}
	\end{equation}
\end{proof}

The following result is a localized isoperimetric inequality for the non-local energy in \eqref{energias<1/2} for $s= \frac{1}{2}$.
\begin{lemma}\label{Ltipoisoper}
	Let $ \Omega \subset \R^n$ be a open bounded with Lipschitz continuous
	boundary and $ \vert \Omega \vert=1$. For every $\eta \in (0,1)$ there exist $t_0>0$ and $ C= C(n,\Omega,\eta)>0$ such that 
	\begin{equation}\label{08062024pom2}
		\inf \left\{ \frac{1}{g_{\frac{1}{2}}(t)} \int_{A} \int_{\Omega \setminus A} P^{\frac{1}{2}}(x-y,t)dxdy\colon A \subset \Omega, \, \vert A \vert \in (\eta, 1-\eta), t \in (0,t_0) \right\} \geq C.
	\end{equation}
\end{lemma}
Before proving
Lemma \ref{Ltipoisoper}, we state the following result which is a consequence of \cite[Theorem 1.4]{AB}.
\begin{lemma}{\cite[Lemma 15]{GM}}\label{lemmaGM}
	Let $\Omega \subset \R^n$ be open and bounded with Lipschitz continuous
	boundary and $ \vert \Omega \vert=1$. Let $ \{\rho_\varepsilon\}_{ \varepsilon \in (0,1)}$ be the standard family of Friedrichs mollifiers with support in $B(0,1)$. For every $\eta \in (0,1)$ there exists $C= C(n,\Omega,\rho, \eta)>0$ such that
	\begin{equation}
		\inf \left\{ \frac{1}{ \varepsilon} \int_{A} \int_{\Omega \setminus A} \rho_\varepsilon(\vert x -y\vert)dx dy\colon A \subset \Omega, \, \vert A \vert \in (\eta, 1-\eta), t \in (0,1) \right\} \geq C.
	\end{equation}
\end{lemma}
The proof of the above lemma it is shown in \cite[Lemma 15]{GM} for $n=2$ and $ \Omega= \left(-\frac{1}{2}, \frac{1}{2}\right)^2$, but the proof in every dimension and for all type of $ \Omega$ with Lipschitz continuous
boundary is fully analogous. We are now in position to prove Lemma \ref{Ltipoisoper}.
\begin{proof}[Proof of Lemma \ref{Ltipoisoper}]
	Fix $\eta \in (0,1)$, $t \in (0,1)$ and let $I \in \N$ be such that $ 2^{-I-1} \leq t \leq 2^{-I}$. We have that
	\begin{equation}\label{06062024matt1}
		\text{if } 0 \leq \vert z \vert \leq 2^{-i}, \, i \in \N \text{ then } \frac{1}{(\vert z \vert^2+t^2)^{\frac{n+1}{2}}} \geq \frac{2^{(n+1)\min\{i,I\}}}{2^{\frac{n+1}{2}}}.
	\end{equation}
Let $\rho$ and $\rho_\varepsilon$ (for $ \varepsilon \in (0,1)$) be as in Lemma \eqref{lemmaGM}. We claim that there exists $ C(n,\rho)>0$ such that
\begin{equation}\label{08062024pom1}
	\sum_{i=1}^{I} 2^i \rho_{2^{-i}}(z) \leq \frac{C(n,\rho)}{(\vert z \vert^2+t^2)^{\frac{n+1}{2}}}  \text{ for every }z \in \R^n.
\end{equation}
Indeed if $ 0 \leq \vert z \vert \leq 2^{-I}$ then we obtain
\begin{equation}
	\begin{split}
		\sum_{i=0}^{I} 2^{i} \rho_{2^{-i}}(z) &\leq \| \rho \|_{\infty} \sum_{i=0}^{I} \big(2^{n+1}\big)^i = \| \rho \|_{\infty} \sum_{i=0}^{I} \frac{\big(2^{n+1}\big)^I}{\big(2^{n+1}\big)^{I-i}} \\
		& \leq \| \rho \|_{\infty} \big(2^{n+1}\big)^{I} \sum_{i=0}^{+\infty} \frac{1}{\big(2^{n+1}\big)^{j}}= \big(2^{n+1}\big)^{I} C(\rho,n) \leq\frac{ C(\rho,n) }{(\vert z \vert^2+t^2)^{\frac{n+1}{2}}}
	\end{split}
\end{equation}
where in the last step we have used \eqref{06062024matt1} for $ i=I$. If $ 2^{-\tilde{i}-1} \leq \vert z \vert \leq 2^{-\tilde{i}}$ for some $\tilde{i}= 0,1,\dots,I-1$, using that $ \rho_{2^{-i}}(z)=0$ for every $i=\tilde{i}+1,\dots,I$ we have 
\begin{equation}
	\begin{split}
	&	\sum_{i=0}^{I} 2^i \rho_{2^{-i}}(z)= \sum_{i=0}^{\tilde{i}} 2^i \rho_{2^{-i}}(z) \leq \| \rho \|_{\infty} \sum_{i=0}^{\tilde{i}} \big(2^{n+1}\big)^{i} \\
	&	\leq \big(2^{n+1} \big)^{\tilde{i}} \| \rho \|_{\infty} \sum_{i=0}^{+\infty} \big(2^{n+1}\big)^{-i}= \big(2^{n+1} \big)^{\tilde{i}} C(\rho,n) \leq \frac{C(\rho,n)}{(\vert z \vert^2+t^2)^{\frac{n+1}{2}}}
	\end{split}
\end{equation}
where the last we have used \eqref{06062024matt1}. If $ \vert z \vert \geq 1$ we have that $ \rho_{2^{-i}}(z)=0$ for all $ i \in 			 \N$ and then 
\begin{equation}
	\sum_{i=0}^{I} 2^i \rho_{2^{-i}}(z) =0 \leq \frac{C(\rho,n)}{(\vert z \vert^2+t^2)^{\frac{n+1}{2}}}.
\end{equation}
Hence we conclude the claim. We prove that \eqref{08062024pom1} implies \eqref{08062024pom2}. We observe that 
$$ \frac{\vert \log(t) \vert}{\log(2)}-1 \leq I \leq \frac{\vert \log(t) \vert}{\log(2)}$$
and hence
\begin{equation}
	\sum_{i=0}^{I} 1= I+1 \geq \frac{\vert \log(t) \vert}{\log(2)}.
\end{equation}
Therefore by \eqref{08062024pom1} and Lemma \ref{lemmaGM} with $\varepsilon$ replaced by $ 2^{-i}$ we obtain
\begin{equation}
	 \begin{split}
	 	\int_{A} \int_{\Omega \setminus A} \frac{1}{(\vert x-y \vert^2+t^2)^{\frac{n+1}{2}}}dx dy & \geq C(\Omega,\rho,n) \sum_{i=0}^{I} 2^i \int_{A} \int_{\Omega \setminus A} \rho_{2^{-i}}(x-y) dx dy\\
	 & \geq C(\Omega,\rho, n,\eta) \sum_{i=0}^{I} 1= C(\Omega,\rho,n,\eta) \vert \log (t) \vert
	 \end{split}
\end{equation}
and hence recalling the very definition of $g_{\frac{1}{2}}(t)$ and of $z \rightarrow P^{\frac{1}{2}}(z,t)$ we obtain \eqref{08062024pom2}.
\end{proof}
We are now in position to prove Theorem \ref{Compthm} when $s = \frac{1}{2}$.
\begin{proof}
	We divide the proof into three steps.\\
	\textit{Step 1.} 
	Let $ \alpha \in (0,1)$ and set $ l_i:= t_i^{\alpha}$ for every $i \in \N$. Let $\{Q_h^i\}_{h \in \N}$ be a disjoint family of cubes of sidelength $l_i$ such that $ \R^n = \cup_{ h \in \N} Q_h^i$. Since $ \vert E_i \vert \leq \vert U \vert$ there exists $H(i) \in \N$ such that, up to permutation of indices,
	\begin{equation}
		\begin{split}
		&\vert Q_h^i \cap E_i \vert \geq \frac{l_i^n}{2} \text{ for every }h=1,\dots, H(i),\\
		&	\vert Q_h^i \setminus E_i \vert > \frac{l_i^n}{2} \text{ for every }h>H(i).
		\end{split}
	\end{equation}
For every $i \in \N$ we set
$$ \tilde{E}_i:= \bigcup_{h=1}^{H(i)}Q_h^i.$$
We claim that there exists a constant $C(n)>0$ such that
\begin{equation}\label{formtes}
	\vert \tilde{E}_i \triangle E_i \vert \leq C(n)Ml_i \vert \log(t_i) \vert  \quad \text{ for  every } i\in \N,
\end{equation}
where $M$ is the constant in \eqref{eq:compacAAA}. Indeed we have
\begin{equation}
	\begin{split}
		\vert E_i \triangle& \tilde{E}_i \vert = \vert \tilde{E}_i \setminus E_i \vert + \vert E_i \setminus \tilde{E}_i \vert = \sum_{h=1}^{ H(i) } \vert Q_h^i \setminus E_i \vert + \sum_{h=H(i)+1}^{+\infty} \vert E_i \cap Q_h^i \vert \\
		& = 2 \sum_{h=1}^{H(i)}\frac{1}{l_i^n} \vert Q_h^i \setminus E_i \vert \frac{l_i^n}{2}+ 2 \sum_{h=H(i)+1}^{+\infty} \frac{1}{l_i^n} \vert E_i \cap Q_h^i \vert \frac{l_i^n}{2}
		 \leq 2 \sum_{h=1}^{+\infty} \frac{1}{l_i^n} \vert Q_h^i \setminus E_i\vert \vert Q_h^n \cap E_i\vert \\
		& \leq 2 \sum_{h=1}^{+\infty} \int_{Q_h^i} \big| \chi_{E_i}(x)- \frac{1}{l_i^n} \int_{Q_h^i} \chi_{E_i}(y)dy    \big| dx  \\
		&\leq C(n) \sum_{h=1}^{+\infty} l_i \int_{Q_h^i\cap E_i} \int_{Q_h^i \setminus E_i} \frac{1}{(\vert x-y \vert^2+t_i^2)^{\frac{n+1}{2}}} dx dy 
		 \leq C(n) l_i \vert \log(t_i) \vert \frac{\mathcal{E}_{t_i}(E_i)}{g_{\frac{1}{2}}(t_i)} \\
		 &\leq C(n)M l_i \vert \log(t_i) \vert 
	\end{split}
\end{equation}
where the second inequality follows by formula \eqref{ambarabacci}, the third inequality is a consequence
of \eqref{05062024pom1}, whereas the last one follows directly by \eqref{eq:compacAAA}. \\
\textit{Step 2.} For every $i \in \N$ let $l_i $ and $ \tilde{E}_i:= \cup_{h=1}^{H(i)}Q_h^i$ be as in \textit{Step 1}. We claim that there exists a constant $C(\alpha,n)$ such that for $i $ large enough
\begin{equation}\label{tesiunouno}
	\mathrm{P}(\tilde{E}_i) \leq C(\alpha,n) \frac{\mathcal{E}_{t_i}(E_i)}{g_{\frac{1}{2}}(t_i)}.
\end{equation} 
We omit the dependence on $i$ by setting $t:= t_i$, $l:=l_i$, $E:=E_i$, $Q_h:=Q_h^i$, $H:= H(i)$, $\tilde{E}:= \tilde{E}_i$. We define the family $\mathcal{R}$ of rectangle $R= \tilde{Q} \cup \hat{Q}$ such that $ \tilde{Q}$ and $ \hat{Q}$ are adjacent cubes (of the type $Q_h$) and $ \tilde{Q} \subset \tilde{E}$, $ \hat{Q} \subset \tilde{E}^c$. We have that
\begin{equation}\label{unouno}
	\begin{split}
		&\mathrm{P}(\tilde{E}) \leq C(n) l^{n-1} \# \mathcal{R},\\
		& \frac{1}{g_{\frac{1}{2}}(t)} \sum_{R \in \mathcal{R}} \int_{R \cap E} \int_{R \setminus E } P^{\frac{1}{2}}(\vert x-y \vert,t) dy dx \leq  \frac{\mathcal{E}_{t}^{\frac{1}{2}}(E)}{g_{\frac{1}{2}}(t)}.
	\end{split}
\end{equation}
By Lemma \ref{Ltipoisoper}, there exists $t_0>0$ such that for every rectangle $\bar{R}$ given by the union of two adjacent unitary cubes in $\R^n$
\begin{equation}\label{unodue}
	\inf \left\{ \frac{1}{g_{\frac{1}{2}}(t)} \int_{F} \int_{\bar{R} \setminus F} P^{\frac{1}{2}}(\vert x-y \vert,t)dy dx: 0 < t <t_0, F \subset \bar{R}, \frac{1}{2}\leq \vert F \vert \leq \frac{3}{2}\right\}:=C(n)>0.
\end{equation} 
For every $A \subset \R^n$ we set $A^l:= \frac{A}{l}$. By formula \eqref{unouno} and \eqref{unodue} with $\bar{R}= R^l$ for every $R \in \mathcal{R}$, for $t $ small enough we have
\begin{equation}
	\begin{split}
		\frac{\mathcal{E}^{\frac{1}{2}}(E)}{g_{\frac{1}{2}}(t)} &\geq \frac{l^{2n}}{g_{\frac{1}{2}}(t)} \sum_{R \in \mathcal{R}} \int_{R^l \cap E^l} \int_{R^l \setminus E^l} P^{\frac{1}{2}}(l\vert x-y \vert,t)dydx \\
		&= \frac{C(n,\alpha)l^{n-1}}{g_{\frac{1}{2}}(\frac{t}{l})} \sum_{R \in \mathcal{R}} \int_{R^l \cap E^l} \int_{R^l \setminus E^l} P^{\frac{1}{2}}(\vert x-y\vert, \frac{t}{l}) dx dy\\
		& \geq C(n,\alpha)l^{n-1}\# \mathcal{R} \geq C(n,\alpha) \mathrm{P}(\tilde{E}),
	\end{split}
\end{equation}
i.e., \eqref{tesiunouno}. \\
\textit{Step 3.} Here we conclude the proof of the compactness result. Fix $ \alpha \in (0,1)$ so that, by \eqref{formtes} 
\begin{equation}\label{fine}
	\vert E_i \triangle \tilde{E}_i \vert \rightarrow0 \text{ as } i \rightarrow +\infty.
\end{equation} 
By \textit{Step 2.} we have that $ \{\tilde{E}_i\}_{i \in \N}$ satisfies the assumption of Theorem \ref{BVcompatness}. Therefore $ \tilde{E}_i \rightarrow E$, up to subsequence, in $\mathrm{L}^1(\R^n)$ with $ \mathrm{P}(E)< +\infty$. Hence by \eqref{fine} we have $ E_i \rightarrow E$ in $\mathrm{L}^1(\R^n)$.
\end{proof}

\section{$\Gamma$-convergence for $s\in (0,\frac12)$} 
The main result of this section is the following.
\begin{theorem}\label{mainthm1}
	Let $s\in (0,\frac{1}{2}) $ and let $ \{t_i\}_{i \in \N} \subset (0,1)$ such that $ t_i \rightarrow 0^+$ as $ i \rightarrow + \infty$. The following $\Gamma$-convergence result holds true.
	\begin{itemize}
		\item[(i)] (Lower bound) Let $E\in \mathrm{M}(\R^n)$ be a Lebesque measurable set. For every $\{E_i\}_{i \in \N} \subset \mathrm{M}(\R^n)$ with $\chi_{E_i}\to\chi_E$ strongly in $\mathrm{L}^1(\R^n)$ it holds
		\begin{equation}\label{trueliminf}
		C_{n,s}\mathrm{P}_{2s}(E)\leq \liminf_{i \rightarrow +\infty}\frac{\mathcal{E}_{t_i}(E_i)}{g_s(t_i)}\,.
		\end{equation}
		\item[(ii)] (Upper bound) For every $E\in\mathrm{M}(\R^n)$ there exists $\{E_i\}_{i \in \N}\subset \mathrm{M}(\R^n)$ such that $\chi_{E_i}\to\chi_E$ strongly in $\mathrm{L}^1(\R^n)$ and
		\begin{equation}\label{limsup1}
		C_{n,s}{\rm P}_{2s}(E)\geq\limsup_{i\to +\infty}\frac{\mathcal{E}_{t_i}(E_i)}{g_s(t_i)}\,.
		\end{equation}
	\end{itemize}
Where $C_{n,s}$ is the constant defined in \eqref{costante!}.
\end{theorem}

\subsection{Proof of Theorem \ref{mainthm1}(i).} We remember that $g_s(t)=t$ for $s \in (0,\frac{1}{2})$. Let $ \{E_i\}_{i \in \N} \subset \mathrm{M}(\R^n)$ such that $ \chi_{E_i} \rightarrow \chi_{E}$ in $\mathrm{L}^1(\R^n)$ and assume $ \liminf_{i \rightarrow +\infty} \frac{\mathcal{E}_{t_i}(E_i)}{t_i} < +\infty$ (otherwise the formula \eqref{trueliminf} is trivial). We claim that
\begin{equation}\label{liminfs<12}
C_{n,s}	{\rm P}_{2s}(E) \leq \liminf_{i \rightarrow +\infty} \frac{\mathcal{E}_{t_i}(E_i)}{t_i}.
\end{equation} 
In fact
\begin{equation}\label{26052024sera1}
	\begin{split}
		\frac{\mathcal{E}_{t_i}(E_i)}{t_i}& = \frac{1}{2} \int_{\R^n} \int_{\R^n} \frac{P^{s}(x-y,t_i)}{t_i} \vert \chi_{E_i}(x)-\chi_{E_i}(y) \vert^2 dx dy \\
		&=  \frac{1}{2} \int_{\R^n} \int_{\R^n} \frac{P^{s}(h,t_i)}{t_i} \vert \chi_{E_i}(x+h)-\chi_{E_i}(x) \vert^2  dx dh\\
		&=  \int_{\R^n} \frac{P^{s}(h,t_i)}{t_i} \int_{\R^n} (1-\cos(\xi \cdot h)) \vert \hat{\chi_{E_i}}(\xi) \vert^2 d\xi dh 
	\end{split}
\end{equation}
where in the last step we have used the Plancherel theorem 
$$ \int_{\R^n} \vert u(x+h)- u(x) \vert^2 dx=  \int_{\R^n} \vert1-e^{i \xi \cdot h} \vert^2 \vert \hat{u}(\xi) \vert^2 d \xi =\int_{\R^n} 2(1-\cos(\xi \cdot h)) \vert \hat{u}(\xi) \vert^2 d \xi.$$
We remember that $ h \rightarrow P^s(h,t)$ is a Schwartz function and $\F [P^s(\cdot, t)] (\xi)= \frac{1}{(2\pi)^{\frac{n}{2}}} e^{-t|\xi|^{2s}}$. We define 
\begin{equation*}
	\mathcal{I}(\xi):= \int_{\R^N}   (1-\cos(\xi \cdot h))P^s(h,t) d h. 
\end{equation*}
We observe that the function $ \mathcal{I}$ is rotationally invariant, that is
\begin{equation}
	\mathcal{I}(\xi)= \mathcal{I}(\vert \xi \vert e_1)
\end{equation}
where $e_1$ denotes the first direction vector in $\R^N$.
Therefore we have 
\begin{equation}\label{26052024sera2}
		\int_{\R^N}( 1-\cos(\xi \cdot h))P^s(h,t) d h =\int_{\R^N}  (1-\cos(\vert\xi\vert h_1))P^s(h,t)d h.
\end{equation}
Now we remember that the Fourier transform of the function $ \alpha \rightarrow \cos(\alpha a)$ for $ a \in \R$ is given by $ \F [\cos( a \cdot)](\beta)= \frac{\sqrt{\pi}}{\sqrt{2}} \left[ \delta(\beta - a)+ \delta(\beta+a)\right]$, this Fourier transform should be read in the sense of the  tempered distributions (that is the dual space of the Schwartz function).
Therefore using the Plancherel theorem we have
\begin{equation}\label{26052024sera3}
	\begin{split}
	&	\int_{\R^n} \frac{P^{s}(h,t)}{t}  (1-\cos(\vert \xi \vert h_1))  dh = \frac{1}{t} \left[1- \int_{\R^n}\cos(\vert \xi \vert h_1) P^s(h,t) dh\right]\\
		&=\frac{1}{t} \left[1-  \int_{\R^{n}}\frac{1}{2} \left( \delta_0(\eta_1 - \vert \xi \vert)+ \delta_0(\eta_1+\vert \xi \vert)\right) \delta_0 (\eta_2) \cdots \delta_0 (\eta_n)  e^{-t \vert \eta \vert^{2s}}d \eta\right] \\
		& = \frac{1}{t} \left[1- e^{-t \vert \xi \vert^{2s} }\right] 
	\end{split}
\end{equation}
where $ \delta_0$ is the one dimensional Dirac delta with center in $ 0 \in \R$.
By formulas \eqref{26052024sera1}, \eqref{26052024sera2}, \eqref{26052024sera3} and for all $R>0$ we obtain
\begin{equation}\label{26052024sera4}
	\begin{split}
	&\liminf_{i \rightarrow +\infty}\frac{\mathcal{E}_{t_i}(E_i)}{t_i} = \liminf_{i \rightarrow +\infty} \int_{\R^n} \vert \F [\chi_{E_i}] (\xi)\vert^2 \frac{1}{t} \left[1- e^{-t_i \vert \xi \vert^{2s} }\right] \\ &\geq \liminf_{i \rightarrow +\infty} \int_{B(0,R)} \vert \F [\chi_{E_i}] (\xi)\vert^2 \frac{1}{t_i} \left[1- e^{-t_i \vert \xi \vert^{2s} }\right] =  \int_{B(0,R)} \vert  \F [\chi_{E}] (\xi)\vert^2 \vert \xi \vert^{2s} d \xi
\end{split}
\end{equation}
where in the last step we have used that $$\lim_{t \rightarrow 0^+}\sup_{\xi \in B(0,R)}\bigg|\frac{1}{t} \left[1- e^{-t \vert \xi \vert^{2s} }\right] - \vert \xi \vert^{2s} \bigg| = 0.$$  Now sending $R \rightarrow +\infty$ in \eqref{26052024sera4} we obtain \eqref{liminfs<12}.
\subsection{Proof of Theorem \ref{mainthm1}(ii).}
Let $ E \in \mathrm{M}(\R^n)$ such that $ \mathrm{P}_{2s}(E) < +\infty$ (otherwise the formula \eqref{limsup1} is trivial) we prove that the limsup in \eqref{limsup1} is realized by the sequence $ E_i=E$ for all $ i \in \N$ and it is a limit.
We remember that $E$ has finite $2s$-fractional perimeter if $\chi_{E}\in \mathrm{H}^s(\R^n)$.
Let $v:\R^n\times [0,+\infty) \rightarrow \R$ be the solution to the Cauchy problem in  \eqref{eq:cauchyproblem} with $u_0= \chi_{E}$
. By the definition of $\E_t (E)$, see \eqref{energia}, and by formula \eqref{solprbCauchy} we have
\[\begin{split}
	\E_t (u_0)=\int_{\R^n}(1-u_0(x)) v(x,t) dx&=\int_{\R^n} v(x, t)\,dx - \int_{\R^n}   v(x, t) u_0(x) \, dx
	\\&= \int_{\R^n} u_0(x)\,dx - \int_{\R^n}   v(x, t) u_0(x) \, dx\\&
	=\int_{\R^n} u_0(x)^2\,dx - \int_{\R^n}   v(x, t) u_0(x) \, dx
\end{split}
\]
where we used that
$$
\int_{\R^n} v(x,t)\, dx= \int_{\R^n}\int_{\R^n} P^s(x-y,t) u_0(y)\, dydx=\int_{\R^n} u_0(y)\, dy
$$
and the fact that $u_0(x)=\chi_{E}(x) \in \{0,1\}$.
Now we observe that 
\[\begin{split}
	\int_{\R^n} u_0(x)^2\,dx - \int_{\R^n}   v(x, t) u_0(x) \, dx=&-\int_{\R^n} u_0(x)\int_0^{t}\pa_h v(x,h)\,dh\,dx
	\\
	=&\int_{\R^n}\int_{0}^{t} u_0(x) (-\Delta)^s v(x,h)\,dh\, dx
\end{split}
\]
where we have used that $ v \in C^1([0,+\infty), \mathrm{L}^2(\R^n)) $.
At this point we just need to use that the Fourier transform is an isometry and
$\F [v(\cdot,t)](\xi)=  \F[u_0(\cdot)] (\xi)  e^{-t|\xi|^{2s}}  $ and therefore
\[
\int_{\R^n}u_0(x) (-\Delta)^s v(x,t)\, dx=
\int_{\R^n} |\xi|^{2s} \frac{1}{(2 \pi)^{\frac{n}{2}}} e^{-t|\xi|^{2s}}  |\F[u_0(\cdot))](\xi)|^2\, d\xi.
\]
Since $u_0 \in \mathrm{H}^s(\R^n)$ we have
\[
\int_{\R^n} |\xi|^{2s} \frac{1}{(2 \pi)^{\frac{n}{2}}}  e^{-t|\xi|^{2s}}  |\F[u_0(\cdot))](\xi)|^2\, d\xi \leq\int_{\R^n} |\xi|^{2s}   |\F[u_0(\cdot))](\xi)|^2\, d\xi= C_{n,s} P_{2s}(E).
\]
Finally to conclude we just need to observe that
\[
\lim_{t\to 0^+} \frac{\E_t(u_0)}{t}= \int_{\R^n} u_0(- \Delta)^s u_0= C_{n,s}  {\rm P}_{2s} (E).
\]
\subsection{Characterization of sets of finite $2s$-fractional perimeter.}
As a byproduct of our $\Gamma$-
convergence analysis, we prove that a set
$E \in \mathrm{M}(\R^n)$ has finite $2s$-fractional  perimeter if and
only if 
$$ \limsup_{t \rightarrow 0^+} \frac{1}{t} \mathcal{E}_t^s(E)< +\infty.$$
\begin{theorem}
	Let $E \subset \mathrm{M}(\R^n)$. The following statement hold true.
	\begin{equation}
		\limsup_{t \rightarrow 0^+} \frac{1}{t} \mathcal{E}_t^s(E)< +\infty \iff {\rm P}_{2s}(E)< +\infty.
	\end{equation}
\end{theorem}
\begin{proof}
	We notice that the implication $ \Longrightarrow$ is a consequence of Theorem \ref{Compthm}. The implication $\Longleftarrow$ is a consequence of Theorem \ref{mainthm1}(ii).
\end{proof}

\section{$\Gamma$-convergence for $s \in [\frac12,1)$.}\label{sec:gammalim}
In this section we threat the remaining cases. The main result of this section is the following.
\begin{theorem}\label{mainthm2}
	Let $s\in [\frac{1}{2},1) $ and let $ \{t_i\}_{i \in \N} \subset (0,1)$ such that $ t_i \rightarrow 0^+$ as $ i \rightarrow + \infty$. The following $\Gamma$-convergence result holds true.
	\begin{itemize}
		\item[(i)] (Lower bound) Let $E\in \mathrm{M}(\R^n)$ be a Lebesgue measurable set. For every $\{E_i\}_{i \in \N} \subset \mathrm{M}(\R^n)$ with $\chi_{E_i}\to\chi_E$ strongly in $\mathrm{L}^1(\R^n)$ it holds
		\begin{equation}\label{trueliminf1}
			\Gamma^{n,s}\mathrm{P}(E)\leq \liminf_{i \rightarrow +\infty}\frac{\mathcal{E}_{t_i}(E_i)}{g_s(t_i)}\,.
		\end{equation}
		\item[(ii)] (Upper bound) For every $E\in\mathrm{M}(\R^n)$ there exists $\{E_i\}_{i \in \N}\subset \mathrm{M}(\R^n)$ such that $\chi_{E_i}\to\chi_E$ strongly in $\mathrm{L}^1(\R^n)$ and
		\begin{equation}\label{limsup2}
			\Gamma^{n,s}{\rm P}(E)\geq\limsup_{i\to +\infty}\frac{\mathcal{E}_{t_i}(E_i)}{g_s(t_i)}\,.
		\end{equation}
	\end{itemize}
	Where
	\begin{equation}\label{09062024pom1}
	\Gamma^{n,s}:= \displaybreak	\begin{cases}
		\displaybreak	\frac{1}{2\pi}\Gamma(1-\frac{1}{2s})\, \qquad &s\in (\frac12,1),
			\\
	\displaybreak	\frac{\omega_{n-2}}{2}	\frac{\Gamma(\frac{n}{2})}{\Gamma(\frac{n+1}{2})} \qquad &s=\frac12
		\end{cases}
	\end{equation}
where $\Gamma(t):= \int_{0}^{+\infty}r^{t-1} e^{-r} dr$ is the Gamma function.
\end{theorem}

Note that the proof of the $\Gamma-$convergence in this case is going to be completely different from the case $ s \in (0,\frac{1}{2})$. We start computing the limit as $t\to 0$ of our functional on hyperlanes.
\subsection{Estimates on cubes}
To begin, we consider the case when $s=\frac12$. This is the only instance where we have a precise formula for the fundamental solution of the fractional heat equation. In what follows we set 
$Q'_\delta=[-\delta, \delta]^{n-1}$,
$Q^+_\delta= [0,\delta]\times Q'$ and $Q^-_\delta= [-\delta,0]\times Q'$.
\begin{lemma}
Let $ \delta>0$. For $t$ small enough it holds
\beq \label{eq:cubes}
\int_{Q_\delta^-}\int_{Q_\delta^+} \frac{t}{(|x-y|^2+t^2)^{\frac{n+1}{2}}}\, dx\,dy\leq c_n \mathcal{H}^{n-1}(Q'_\delta) t\vert  \log t \vert + o(t\vert \log t \vert ),
\eeq
where $c_n$ is given in \eqref{eq:cn}
\end{lemma}
\begin{proof}
Using Fubini Tonelli Theorem we have
\begin{multline}
	\int_{Q_\delta^-}\int_{Q_\delta^+} \frac{t}{(|x-y|^2+t^2)^{\frac{n+1}{2}}}\, dx\,dy
	\\=\int_{Q'_\delta}\int_{0}^\delta\int_{-\delta}^0 \int_{Q'_\delta} \frac{t}{(|x'-y'|^2+(x_n-y_n)^2+t^2)^{\frac{n+1}{2}}}\,dy'\,dx_n\,dy_n\,dx'.
\end{multline}
We observe that
\beq \label{eq:fubini}
\begin{split}
\int_{0}^\delta\int_{-\delta}^0 \int_{\R^{n-1}}  P^{\frac{1}{2}}(x-y,t)\, dx_n dy_n dy' 
=c_n t
\int_{0}^\delta\int_{-\delta}^0 \frac{1}{(x_n-y_n)^2+t^2}\, dx_n dy_n
\end{split}
\eeq
where we used the change of variable 
$$
z=\frac{y'-x'}{\sqrt{(x_n-y_n)^2+t^2}}
$$
and we have set
\beq \label{eq:cn}
c_n:=\int_{\R^{n-1}} \frac{1}{(1+|z|^2)^\frac{n+1}{2}}\, dz.
\eeq
A straightforward computation gives
\beq\label{eq:cubes1}\begin{split}
t \int_{0}^\delta\int_{-\delta}^0 \frac{1}{(x_n-y_n)^2+t^2}\, dx_n dy_n
&=  2\delta ( \arctan \frac{t}{\delta} - \arctan \frac{t}{2\delta}   )
-t\log t +t \log\left(   \frac{\delta^2+t^2}{t^2+4\delta^2}  \right)
\\
&:= g(t,\delta) + t|\log t|
\end{split}
\eeq
with
\begin{equation}
c_1(\delta)t \leq g(t,\delta)\leq t c_2(\delta)
\end{equation}
where $ c_1(\delta), c_2(\delta) \rightarrow 0$ as $ \delta \rightarrow 0^+$. Therefore, to conclude we just need to observe that by monotoniticity of the integral and \eqref{eq:cubes1} we have
\begin{equation}
	\begin{split}
\int_{Q_\delta^-}\int_{Q_\delta^+} &\frac{t}{(|x-y|^2+t^2)^{\frac{n+1}{2}}}\, dx\,dy
\\
&\leq \int_{Q_\delta^-}\int_{\R^{n-1}} \frac{t}{(|x-y|^2+t^2)^{\frac{n+1}{2}}}\, dx\,dy
= c_n t|\log t|+o(t\log t)
\end{split}
\end{equation}
which is exactly \eqref{eq:cubes}.
\end{proof}
We now estimate the integral on cubes by above.
\begin{lemma}
Let $c_n$ be the constant given in \eqref{eq:cn}. For any $\delta>0$ and for $t$ small enough it holds true
\[
\int_{Q_\delta^-}\int_{Q_\delta^+} \frac{t}{(|x-y|^2+t^2)^{\frac{n+1}{2}}}\, dx\,dy\geq c_n\Ha^{n-1}(Q_\delta)  (t \vert \log t  \vert + o(t\log t)).
\]
\end{lemma}
\begin{proof}
To prove this lemma we just need to show that the error passing from a small cube to the hyperplane $\R^{n-1}$ is negligible when $t$ is small. To do that we start decomposing the integral
\begin{equation}
	\begin{split}
		&\int_{0}^\delta\int_{-\delta}^0 \int_{Q'_\delta} P^{\frac{1}{2}}(x-y,t)\, dx_ndy'd y_n
		\\
		&=
		\int_{0}^\delta\int_{-\delta}^0\left( \int_{\R^{n-1}} P^{\frac{1}{2}}(x-y,t)\, dy' -
		\int_{(Q'_\delta)^c} P^{\frac{1}{2}}(x-y,t)dy' \right) \, dx_ndy_n
	\end{split}
\end{equation}
and observe that by a change of variable $ z=\frac{y'-x'}{\sqrt{(x_n-y_n)^2+t^2}}$ we have
\begin{multline}
	\int_{0}^\delta\int_{-\delta}^0 \int_{\R^{n-1}\setminus Q'_\delta} P^s(x-y,t)\, dx_ndy_n dy'
	\\=
	\frac{1}{t}\int_{0}^\delta\int_{-\delta}^0 \int_{\R^{n-1}\setminus Q'_{\frac{\delta}{t}}(x')} \frac{1}{\left(|z|^2+\frac{(x_n-y_n)^2}{t^2}+1\right)^{\frac{n+1}{2}}}\,dz\,dx_n\,dy_n.
\end{multline}
Using the integration in polar coordinate we obtain
\beq \label{eq:cubes2}
\begin{split}
\int_{0}^\delta\int_{-\delta}^0 \int_{\R^{n-1}\setminus Q'_{\frac{\delta}{t}}(x')} & \frac{\,dz\,dx_n\,dy_n}{\left(|z|^2+\frac{(x_n-y_n)^2}{t^2}+1\right)^{\frac{n+1}{2}}}
\leq 
\int_{0}^\delta\int_{-\delta}^0 \int_{\R^{n-1}\setminus Q'_{\frac{\delta}{t}}(x')} \frac{dz\,dx_n\,dy_n }{\left(|z|^2+1\right)^{\frac{n+1}{2}}}\,
\\
&\leq 
(n-1)\omega_{n-1}\int_{0}^\delta\int_{-\delta}^0 \int_{\frac{\delta}{t} }^{\infty} \frac{\rho^{n-2}}{\left(\rho^2+1\right)^{\frac{n+1}{2}}}\,d\rho \,d x_n\, dy_n
\\
&\leq
(n-1)\omega_{n-1}\delta^2\int_{\frac{\delta}{t^2} }^{\infty} \frac{\rho}{(\rho^2+1)^2}\,d\rho \\&
=\frac{(n-1)\omega_{n-1}\delta^2 }{\delta^2+t^2}  \, t^2 . 
\end{split}
\eeq
The conclusion now follows by combining \eqref{eq:fubini}, \eqref{eq:cubes1} and \eqref{eq:cubes2}.
\end{proof}
We now procede performing the same computation for $s>\frac12$. We stress that in this instance we do not have the precise expression of the kernel, as it is not known. However,
 the Fourier transform of the function $P^s(z,t)$ is explicit and can be used to perform the computations that we need. In what follows we set $ \R_+^{n}:= \R^{n-1} \times [0,+\infty)$ and $R_\delta:=\{(x',x_n): x'\in \R^{n-1}, x_n\in (0,\delta)\}$ for all $ \delta>0$, 
\begin{lemma}\label{lem:sbigger}
Let $\delta>0$, $s \in(\frac12,1)$. Then for $t$ small enough it holds true
\beq \label{eq:cubes3}
\int_{Q_\delta^-} \int_{\R_+^{n}} P^s(x-y,t)\,dx\,dy=\frac{1}{2\pi}\Gamma\left(1-\frac{1}{2s}\right) t^\frac{1}{2s}(\Ha^{n-1}(Q_\delta) +o(1))   .
\eeq
\end{lemma}
\begin{proof}
We start by observing that
	\begin{equation}
		\int_{Q_\delta^-} \int_{\R_+^{n}\setminus R_\delta} P^s(x-y,t)\,dx\,dy= o(1).
	\end{equation}
We claim that
\beq\label{eq:fourier1}
\int_{\R^{n-1}}P^s((z',z_n),t)\, dz'= 2\int_0^\infty e^{-t|r|^{2s}}\cos (z_n r)\, dr.
\eeq
For all $ z_n \in \R$ we set $v_{z_n}(z',t):=P^s((z',z_n),t)$. We have
\begin{equation}
	\begin{split}
		P^s(z,t)= \F^{-1}[\frac{1}{(2 \pi)^{\frac{n}{2}}}e^{-t|\cdot|^{2s}}](z)&= \int_{\R^n}\frac{1}{(2 \pi)^{n}} e^{-t|\xi|^{2s}}e^{i\langle z,\xi\rangle}\, d\xi
	  \\&=
		\int_{\R^{n-1}}\frac{1}{(2 \pi)^\frac{n-1}{2}}e^{i \langle z',\xi'\rangle} \int_{\R} \frac{1}{(2 \pi)^{1+\frac{n-1}{2}}}e^{-t|\xi|^{2s}}e^{i z_n \xi_n}\, d\xi_n\, d\xi.
	\end{split}
\end{equation}
Since the Fourier transform is an isometry we have
\[
v_{z_n}(z',t)= \F^{-1}[\F[v_{z_n}(\cdot,t)]](z')= \frac{1}{(2 \pi)^\frac{n-1}{2}} \int_{\R^{n-1}}e^{i \langle z', \xi'\rangle}  \F[ v_{z_n}  (\cdot, t)](\xi')    \, d\xi',
\]
where with a little abuse of notations we denoted the Fourier transform in $\R^{n-1}$ still by $\F[\cdot](\cdot)$. Therefore by uniqueness of the Fourier transform 
 we arrive at
\[
\begin{split}
\F[ v_{z_n}  (\cdot, t) ](\xi') &=\frac{1}{(2 \pi)^{1+\frac{n-1}{2}}}\int_{\R} e^{-t|\xi|^{2s}}e^{i z_n \xi_n}\, d\xi_n
\\
&= \frac{1}{(2 \pi)^{1+\frac{n-1}{2}}}\int_{\R} e^{-t|\xi|^{2s}}\cos (z_n \xi_n)\, d\xi_n\\
&= \frac{1}{(2 \pi)^{1+\frac{n-1}{2}}}\int_{\R} e^{-t(\vert\xi'\vert^2+\xi_n^2)^{s}}\cos (z_n \xi_n)\, d\xi_n.
\end{split}
\]
From the above formula and by
\[
\begin{split}
\int_{\R^{n-1}}&P^s((z',z_n),t)\, dz'= \int_{\R^{n-1}}v_{z_n}(z',t)\, dz'\\
&=(2 \pi)^\frac{n-1}{2}\F[ v_{z_n}  (\cdot, t) ](0)
=\frac{1}{\pi}\int_0^\infty e^{-t|\xi_n|^{2s}}\cos (z_n \xi_n)\, d\xi_n,
\end{split}
\]
we obtain  \eqref{eq:fourier1}. 
\\ Now using Fubini Theorem and  \eqref{eq:fourier1} with $x_n-y_n$ instead of $z_n$ we obtain that
\beq\label{eq:cubes4}
\begin{split}
\int_{Q_\delta^-} \int_{R_{\delta}} P^s(x-y,t)\,dx\,dy=&
\int_{Q_\delta}\int_{-\delta}^0 \int_{\R^{n-1}}\int_0^{\delta} P^s((x'-y',x_n-y_n),t)\,dx'\,dy'\, dx_n\, dy_n\\
=&\frac{1}{\pi}
|Q_\delta| \int_{-\delta}^0\int_{0}^{\delta} \int_0^\infty e^{-tr^{2s}}\cos (r(x_n-y_n)) )\, dr\, dx_n\,dy_n.
\end{split}
\eeq
Now we use basic trigonometry to have
\[
\int_{-\delta}^0\int_{0}^{\delta}\cos (r(x_n-y_n)) )\, dr\, dx_n\,dy_n
=\frac{\sin^2 (\delta r)   -(1-\cos (\delta r))^2   }{r^2}
=2 \cos(\delta r) \frac{(1 -\cos (\delta r))   }{r^2}
\]
Thus we have that \eqref{eq:cubes4} becomes
\[
\int_{Q_\delta^-} \int_{\R^n_+} P^s(x-y,t)\,dx\,dy=
\frac{2}{\pi}|Q_\delta|  \int_0^\infty e^{-tr^{2s}} \cos(\delta r) \frac{(1 -\cos (\delta r))   }{r^2}\, dr
\]
Now we observe that
\[\begin{split}
 \int_0^\infty e^{-tr^{2s}} \cos(\delta r)& \frac{(1 -\cos (\delta r))   }{r^2}\, dr \\
&=  \int_0^\infty e^{-tr^{2s}}  \frac{(\cos (\delta r))  -1 }{r^2}\, dr
+ \int_0^\infty e^{-tr^{2s}} \frac{(1 -\cos^2 (\delta r))   }{r^2}\, dr
\\
&=
\int_0^\infty e^{-tr^{2s}}  \frac{(\cos (\delta r))  -1 }{r^2}\, dr
+\frac{1}{2} \int_0^\infty e^{-tr^{2s}} \frac{(1 -\cos (2\delta r)  }{r^2}\, dr
\\&=
\int_{0}^\infty 
 e^{-tr^{2s}}  \frac{(\cos (\delta r))  -1 }{r^2}\, dr
+\int_0^\infty e^\frac{-tr^{2s}}{2^{2s}}  \frac{(1 -\cos (\delta r)  }{r^2}\, dr
\\
&= 
\int_ {0}^{\infty}
\frac{(1 -\cos (\delta r)  }{r^2}   (         e^\frac{-tr^{2s}}{2^{2s}}  -  e^{-tr^{2s}})              \, dr 
= t^{\frac{1}{2s}} h(t)
\end{split}
\]
where 
$$h(t):=\int_0^\infty\frac{1 -\cos \left(\frac{\delta r}{t^{\frac{1}{2s}}}\right)  }{r^2}   (         e^\frac{-r^{2s}}{2^{2s}}  -  e^{-r^{2s}})              \, dr >0.
$$
It is quite easy to show that $h(t)$ is bounded  and  by applying Riemann-Lebesgue lemma we finally have
\[
\lim_{t\to 0} h(t)= \int_ {0}^\infty 
\frac{  e^\frac{-r^{2s}}{2^{2s}}  -  e^{-r^{2s}}}{r^2}              \, dr 
=\frac{1}{2} \int_{0}^{\infty}\frac{1-e^{-r^{2s}}}{r^2}\, dr=
\frac{1}{4s} \int_{0}^{\infty}\frac{1-e^{-r}}{r^{1+\frac12s}}\, dr=\frac{\Gamma(1-\frac{1}{2s})}{4}
\]
which readly gives \eqref{eq:cubes3}.
\end{proof}
The next lemma is an obvious consequence of the previous one.
\begin{lemma}
Let $Q_\delta^+. Q_\delta^-$ and $Q_\delta$ as above. For  $s \in (\frac{1}{2},1)$ and for $t$ small enough it holds true
\beq \label{eq:cubes5}
\int_{Q_\delta^-} \int_{Q_\delta^+} P^s(x-y,t)\,dx\,dy\leq\frac{\Gamma\left(1-\frac{1}{2s}\right) }{ 2 \pi}t^\frac{1}{2s}(\Ha^{n-1}(Q_\delta)  +o(1))  . 
\eeq
\end{lemma}
Finally we prove the crucial estimate to prove the $\Gamma$ convergence result.
\begin{lemma}
Let $Q_\delta^+. Q_\delta^-$ and $Q_\delta$ as above. For  $s>\frac12$
we have
\beq \label{eq:cubes6}
\int_{Q_\delta^-} \int_{Q_\delta^+} P^s(x-y,t)\,dx\,dy\geq\frac{\Gamma\left(1-\frac{1}{2s}\right) }{ 2 \pi}t^\frac{1}{2s}(\Ha^{n-1}(Q_\delta) +o(1))   .
\eeq
\end{lemma}
\begin{proof}
To prove this lemma we just need to show that the error passing from a small cube to the hyperplane $\R^{n-1}$ is negligible when $t$ is small. To do that we start decomposing the integral
\begin{equation}
	\begin{split}
		\int_{0}^\delta &\int_{-\delta}^0 \int_{Q'_\delta} P^s(x-y,t)\, dx_ndy
		\\
		&=
		\int_{0}^\delta\int_{-\delta}^0 \int_{\R^{n-1}} P^s(x-y,t)\, dx_ndy -
		\int_{0}^\delta\int_{-\delta}^0 \int_{\R^{n-1}\setminus Q'_\delta} P^s(x-y,t)\, dx_ndy
	\end{split}
\end{equation}

and observe that by \eqref{eq:scaling} and using a change of variable we have
\[\begin{split}
\int_{0}^\delta\int_{-\delta}^0  \int_{\R^{n-1}\setminus Q'_\delta} P(x-y,t)\,& dx_ndy
=t^{-\frac{n}{2s}}\int_{0}^\delta\int_{-\delta}^0 \int_{\R^{n-1}\setminus Q'_\delta} P^s\left(\frac{x-y}{t^{\frac{1}{2s}}},1\right)\, dx_ndy
\\&=
\frac{1}{t^\frac{1}{2s}}\int_{0}^\delta\int_{-\delta}^0 \int_{\R^{n-1}\setminus Q'_{\frac{\delta}{t}}(x')}  P^s\left( z, \frac{y_n-x_n}{t^\frac{1}{2s}},1\right)  \,dz\,dx_n\,dy_n.
\end{split}
\]
Now we using \eqref{eq:decay} we have that
\beq \label{eq:cubes7}
\begin{split}
t^{-\frac{1}{2s}}\int_{0}^\delta\int_{-\delta}^0 \int_{\R^n\setminus Q'_{\frac{\delta}{t^{\frac{1}{2s}}}}(x')}&  P^s\left( z, \frac{y_n-x_n}{t^\frac{1}{2s}},1\right) \,dz\,dx_n\,dy_n
\\
&\leq 
t^{-\frac{1}{2s}}\int_{0}^\delta\int_{-\delta}^0 \int_{\R^n\setminus B_{\frac{\delta}{t^{\frac{1}{2s}}}}(x')} \frac{1}{|z|^{n+2s}}\,dz\,dx_n\,dy_n
\\
&\leq (n-1)\omega_{n-1}
t^{-\frac{1}{2s}} \int_{0}^\delta\int_{-\delta}^0 \int_{\frac{\delta}{t^{\frac{1}{2s}}} }^{\infty} \rho^{-2-2s}\,d\rho \,d x_n\, dy_n
\\
&
= (n-1)\omega_{n-1} \delta^{1-2s} t=o(t^\frac{1}{2s})
\end{split}
\eeq
The proof of \eqref{eq:cubes6} is now direct consequence of Lemma \ref{lem:sbigger} and inequality \eqref{eq:cubes7}.
\end{proof}
Now we prove that $ \Gamma^{n,\frac{1}{2}}$ defined in \eqref{09062024pom1} is equal to the constant $c_n$ defined in \eqref{eq:cn}.
\begin{lemma}\label{conto}
	Let $a,b\in \R$ and $k \in \N$ such that $b>-k$ and $a>k+b$, then
	\begin{equation}\label{contounouno}
		\int_{\R^k} \frac{\vert x \vert^b}{(1+\vert x \vert^2)^{\frac{a}{2}}} dx= \frac{\omega_{k-1}}{2}
		\frac{ \Gamma(\frac{a+k}{2}) \Gamma(\frac{a-b-k}{2})}{\Gamma(\frac{a+b}{2})}.
	\end{equation}
	In particular 
	\begin{equation}\label{contounodue}
		c_n:= \int_{\R^{n-1}} \frac{1}{(1+\vert x \vert^2)^{\frac{n+1}{2}}} dx= \frac{\omega_{n-2}}{2}	\frac{\Gamma(\frac{n}{2})}{\Gamma(\frac{n+1}{2})} := \Gamma^{n,\frac{1}{2}}.
	\end{equation}
\end{lemma}
Before to prove Lemma \eqref{conto} we recall the definition of Euler’s beta function
\begin{equation}
	B: (0,+\infty)\times (0,+\infty) \rightarrow \R, \quad B(x,y):= 2 \int_{0}^{\frac{\pi}{2}} (\cos(\theta))^{2x-1} (\sin(\theta))^{2y-1} d \theta.
\end{equation}
The relation between the beta and the gamma function is the following
\begin{equation}\label{betafunct}
	B(x,y)= \frac{\Gamma(x)\Gamma(y)}{\Gamma(x+y)}.
\end{equation}
\begin{proof}[Proof of Lemma \ref{conto}]
	We observe that the function 
	$ h(x):=\frac{\vert x \vert^b}{(1+\vert x \vert^2)^{\frac{a}{2}}}$ belong $\mathrm{L}^1(\R^k)$. Integrating in polar coordinate and by \eqref{betafunct} we have that
	\begin{equation}
		\begin{split}
			\int_{\R^k} &\frac{ \vert x \vert^b}{(1+\vert x \vert^2)^{\frac{a}{2}}} dx= \omega_{k-1} \int_{0}^{+\infty} \frac{ \rho^{b+k-1}}{(1+\rho^2)^{\frac{a}{2}}}d \rho \\
			& 
			\overset{\rho= \tan(\eta) }{=} \omega_{k-1} \int_{0}^{\frac{\pi}{2}} \frac{ (\tan(\eta))^{b+k-1}}{(1+ \tan^2(\eta))^{\frac{a-2}{2}}} d \eta 
			= \omega_{k-1} \int_{0}^{\frac{\pi}{2}} (\sin(\eta))^{b+k-1} (\cos(\eta))^{a-b-k-1} d \eta
			\\
			&= \frac{\omega_{k-1}}{2} B(\frac{b+k}{2}, \frac{a-b-k}{2})=  \frac{\omega_{k-1}}{2}
			\frac{ \Gamma(\frac{a+k}{2}) \Gamma(\frac{a-b-k}{2})}{\Gamma(\frac{a+b}{2})}
		\end{split}
	\end{equation}
	hence \eqref{contounouno}. Formula \eqref{contounodue} follows by \eqref{contounouno} with $b=0$ and $k=n-1$ and $a=n+1$.
\end{proof}

In the next corollary we summarize all the previous result of this sections.
\begin{corollary} \label{cor:cubes}
Let 
$g_s$ be the function defined in \eqref{eq:definizioneg} and
\beq \label{eq:defgamma}
\Gamma^{n,s}=
\begin{cases}
(2 \pi)^{n-2}\Gamma(1-\frac{1}{2s})\, \qquad &s\in (\frac12,1),
\\
\frac{\omega_{n-2}}{2}	\frac{\Gamma(\frac{n}{2})}{\Gamma(\frac{n+1}{2})} \qquad &s=\frac12.
\end{cases}
\eeq
Then
\[
\lim_{t\to 0} \frac{1}{g_s(t)} \int_{Q^-_\delta}\int_{Q_\delta^+}P^s(x-y,t)\, dxdy= \Gamma^{n,s}\Ha^{n-1}(Q_\delta).
\]
\end{corollary}
We are now ready to prove our $\Gamma-$convergence theorem.
\subsection{Proof of Theorem \ref{mainthm2}(i).}
\begin{proposition}
Let $s\in[ \frac12,1)$, $Q=[-1,1]^n$,
$g_s(t)$ defined in \ref{cor:cubes}
and let $E$ a set of finite perimeter.
Set
\beq\label{eq:liminf2}
\Gamma_{n,s}= \inf \left\{ \liminf_{t\to 0}  \frac{1}{g_s(t)} \int_{E_{t}^c \cap Q}\int_{E_{t} \cap  Q}P^{s}(x-y,t)\, dxdy: \, E_{t}\to \R_{+}^n \text{ in } \mathrm{L}^1(\R^n)  \right\}.
\eeq
Then for any sequence $E_i\to E$ and $t_i\to 0$ we have
\beq\label{eq:liminf3}
\liminf_i \frac{1}{g_s(t_i)} \int_{E_i}\int_{E_i^c} P^s(x-y,t)\, dxdy \geq 
\Gamma_{n,s} {\rm P}(E)
\eeq
with the constant $\Gamma_{n,s}$ defined in \eqref{eq:liminf2}.
\end{proposition}
\begin{proof}
The result follows by following the strategy a well estabilished strategy based on blow up, introduced in this context in\cite{FM}.
Denote by ${\mathcal C}$ the family of all $n$-cubes in $\R^1{n}$
$$
{\mathcal C}:=\left\{R(x+rQ):\ x\in\R^{n},\,\,r>0,\,\,R\in
SO(n)\right\}.
$$
Since the set $E$ is of finite perimeter 
$|\nabla \chi_E|$-a.e. $x_0\in \R^n$ there exists $R_{x_0}\in SO(n)$ such that
$(E-x_0)/r$ locally converge in measure to $R_{x_0}H$ as $r\to 0$ and 
\begin{equation}\label{eq:densitycube}
\lim_{r\to 0}\frac{|\nabla \chi_E|(x_0+rR_{x_0}Q)}{r^{n-1}}=1,\quad \text{for  $|\nabla \chi_E|$-a.e. } x_0.
\end{equation}
Now, given a cube $C\in{\mathcal C}$ contained in $\Omega$ we set
$$\alpha_i(C):=
\frac{1}{g(t_i)} \int_{E_i^c \cap C}\int_{E_i \cap C}P^s(x-y,t_i)\, dxdy
$$
and
$$\alpha (C):=\liminf_{i\to\infty}\alpha_i(C).$$
We claim that, setting $C_r(x_0):=x_0+rR_{x_0}Q$, where $R_{x_0}$ is as in
\eqref{eq:densitycube}, for $|\nabla \chi_E|$-a.e. $x_0$ we have
\begin{equation}\label{eq:claim1}
\Gamma_{n,s} \le \liminf_{r\to 0}\frac{\alpha(C_r(x_0))}{r^{n-1}} \quad
\text{for $|\nabla \chi_E|$-a.e. }x\in \R{n}.
\end{equation}

Since from now on $x_0$ is fixed, we assume $R_{x_0}=I$, so that the limit
hyperplane is $H$ and the cubes $C_r(x_0)$ are the standard ones
$x_0+rQ$. Let us choose a sequence $r_k\to 0$ such that
$$\liminf_{r\to 0}\frac{\alpha(C_r(x_0))}{r^{n-1}}=
\lim_{k\to \infty} \frac{\alpha(C_{r_k}(x_0))}{r_k^{n-1}}.$$ For $k>0$
we can choose $i(k)$ so large that the following conditions hold:
\begin{equation*}
\begin{cases}
&\alpha_{i_k}(C_{r_k}(x_0))\leq \alpha(C_{r_k}(x_0))+r_k^n,\\
&\int_{C_{r_k}(x)}|\chi_{E_{i_k}}-\chi_E|dx<\frac{|C_{r_k}|}{k}.
\end{cases}
\end{equation*}
Then we infer
\beq \label{eq:liminf}
\begin{split}
\frac{\alpha(C_{r_k}(x_0))}{r_k^{n-1}}&\geq \frac{\alpha_{i(k)}(C_{r_k}(x_0))}{r_k^{n-1}}-r_k\\
&=
\frac{1}{g(t_{i_k}) r_{k}^{n-1} } \int_{E_{i_k}^c \cap C_{r_k}(x_0)}\int_{E_{i_k} \cap  C_{r_k}(x_0)}P^s(x-y,t_{i_k})\, dxdy -r_k.
\end{split}
\eeq
Now we distinguish the case $s=1/2$ and $s\in (1/2,1)$. In the former case the above inequality reads as
\beq\label{eq:caseoneliminf}
\frac{\alpha(C_{r_k}(x_0))}{r_k^{n-1}}\geq  \frac{1}{|\log t_{i_k}|} \int_{(E_{i_k}-x_0)/r_k \cap C }\int_{(E_{i_k}-x_0)/r_k \cap   C}\frac{1}{(|x-y|^2+\frac{t_{i_k}^2} {r_{k}^2})^\frac{n+1}{2}}\, dxdy
\eeq
Up to extract a further subsequence, we assume that
$t_{i_k}\leq r_k^{k+1}$
and thereofore
\beq
\log\frac{1}{t_{i_k}}
= \log\frac{r_k}{t_{i_k}}+\frac{1}{k+1 }\log\frac{1}{r^{k}_k} \leq
 \log\frac{r_k}{t_{i_k}}+\frac{1}{k }\log\frac{1}{r^{k}_k}
\leq (1+\frac1k)\log\frac{r_k}{t_{i_k}}
\eeq
Since $(E_{i_k}-x)/r_k\to H$ the above inequality implies
joint with \eqref{eq:caseoneliminf} provides
\[
\liminf_{r\to 0} \frac{\alpha(C_r(x_0))}{r_k^{n-1}}\geq \Gamma_{n,s}
\]
where the constant $\Gamma_{n,s}$ is defined in \eqref{eq:defgamma}
\\
{\it Case two:} $s>\frac12$.
\\
In this case we recall that $g(t)=t^\frac{1}{2s}$ and then using \eqref{eq:scaling} we get
\[\begin{split}
\frac{\alpha(C_{r_k}(x_0))}{r_k^{n-1}}
&=
\frac{  t_{i_k}^{-\frac{n+1}{2s}}  }{r_{k}^{n-1}} \int_{E_{i_k}^c \cap C_{r_k}(x_0) }\int_{E_{i_k}^ \cap C_{r_k}(x_0)}P^s((x-y)t^{-\frac{1}{2s}}, 1)\, dxdy -r_k.
\\
&=
   (t_{i_k}^{-\frac{1}{2s}}  r_k)    ^{n+1}
 \int_{(E_{i_k}-x_0)/r_k \cap C }\int_{(E_{i_k}-x_0)/r_k \cap  C )}P^s((x-y)r_k t_{i_k}^{-\frac{1}{2s}},1)\, dxdy -r_k
\\
&
= (t_{i_k}   r_k^{-2s})^{-\frac{1}{2s}}  
 \int_{(E_{i_k}-x_0)/r_k \cap C }\int_{(E_{i_k}-x_0)/r_k \cap  C )}P^s((x-y),t_{i_k}r_k^{-2s})\, dxdy -r_k
\end{split}
\]
Therefore, choosing $i_k$ so large that $t_{i_k}r_{k}^{-2s}\to 0$
we get
\[
\frac{\alpha(C_{r_k}(x_0))}{r_k^{n-1}}
\geq \Gamma_{n,s}
\]
with $\Gamma_{n,s}$ defined in \eqref{eq:liminf2}.
\end{proof}
We are now in position to prove Theorem \ref{mainthm2}(i).
\begin{proof}[Proof of Theorem \ref{mainthm2}(i)]
	To conclude the proof we need to show that $\Gamma_{n,s}=\Gamma^{n,s}$. The proof of this fact is the same as in \cite{ADPM}[Lemma 13]. We decide to not reproduce it here in deteails but to just sketch it.
	First, via a gluing argument, which is possibile since a coarea formula holds for our functional, one can show that the infimum in the definition of $\Gamma_{n,s}$ can be take among sets which coincide with $H$ outside a smaller cube $Q_\delta$ and then, via a nonlocal calibration argument, one also gets the minimality of the halfspace (see \cite{C,P}), which readily provides $\Gamma^{n,s}=\Gamma_{n,s}$.
\end{proof}
\subsection{Proof of Theorem \ref{mainthm2}(ii).}
This subsection we prove Proof of Theorem \ref{mainthm2}(ii).
\begin{proof}[Proof of Theorem \ref{mainthm2}(ii)]
We start by proving that for a polyhedron $\Pi$ it holds
\beq\label{eq:polyhedron}
\limsup_{t\to 0} \frac{1}{g_s(t)} \int_{\Pi}\int_{\Pi^c}P^s(x-y,t)\,dxdy\leq  \Gamma^{n,s} {\rm P}(\Pi) .
\eeq

 For any $\e>0$ there exists $\delta_0$ such that 
for any $\delta \in (0,\delta_0)$ there exists a collection of $N_\delta$ cubes of volume $(2\delta)^n$ centered at $x_i\in \pa \Pi$ such that each cubes $Q_\delta(x_i)$ intersect one and only one face of $\pa \Pi$ 
and
\begin{itemize}
\item
if $\tilde Q_\delta(x_i)$ denotes the dilation of $Q_\delta(x_i)$ by a factor
$(1+\delta)$, then each cube $\tilde Q_\delta(x_i)$ intersects exactly one
face $\Sigma$ of $\pa\Pi$ and
each of its sides is either parallel or orthogonal to $\Sigma$;
\item  $|\Ha^{n-1}(\pa \Pi)   -      N_\delta (2\delta)^{n-1} |<\e$.
\end{itemize}
Moreover, since we are interested in the limit as $t\to 0$ we can assume that $t<\delta$. 
For all $ \delta>0$ we set
\begin{equation}
	(\partial \Pi)_{\delta}:= \left\{x \in \R^n : d(x, \partial \Pi) < \delta\right\}, \quad (\partial \Pi)_{\delta}^{-}:= (\partial \Pi)_{\delta} \cap \Pi.
\end{equation}
For $x \in \Pi$ we define
\[
I(x,t)= \int_{\Pi^c}P^s(x-y,t)\, dy.
\]
We now have to distinguish among many cases.\\
{\it Case 1:} $x \in \Pi \setminus (\Pi)_\delta^-$.\\
In this case
we note that \eqref{eq:decay} implies the existence of a constant $c_{n,s}$ such that
\[
P^s(x-y,t)\leq \frac{c_{n,s} t  }{|x-y|^2+t^\frac{1}{s})^\frac{n+2s}{2}}
\]
and therefore
\beq \label{eq:case1}
\begin{split}
I_s(x,t)\leq c_{n,s} \int_{\Pi^c }\frac{t}{(|x-y|^2+t^\frac1s)^\frac{n+2s}{2}}&\leq 
n\omega_n t\int_{\delta}^\infty \frac{\rho^{n-1}}{(\rho^2+ t^\frac1s)^\frac{n+2s}{2}}\,d\rho
\\
&\leq n\omega_n  t\int_{\delta}^{\infty}\frac{\rho}{(\rho^2+t^\frac1s)^{1+s}}\, d\rho
\\
&= \frac{\delta t}{2s(\delta^2+t^\frac{1}{s})^{s}}.
\end{split}
\eeq
{\it Case two:} $x\in \Pi_\delta^- \setminus \bigcup_i Q_\delta(x_i)$.
\\
In this case we
 write $\pa\Pi=\bigcup_{j=1}^J\Sigma_j$,  and define
$$(\pa\Pi)^-_{\delta,j}:=\{x\in (\pa\Pi)^-_\delta:d(x,\Pi^c)=d(x,\Sigma_j)\}.$$
Clearly $(\pa\Pi)^-_\delta=\bigcup_{j=1}^J (\pa\Pi)^-_{\delta,j}$ and 
$$(\pa\Pi)^-_{\delta,j}\subset\{x+t\nu:x\in \Sigma_{\delta,j},\, t\in (0,\delta),\, \nu\text{ is the interior unit normal to }\Sigma_{\delta,j}\}, $$
and $\Sigma_{\delta,j}$ is the set of points $x$ belonging to the same hyperplane as $\Sigma_j$ and with $d(x,\Sigma_j)\leq \delta$.  Set 
$d(x):= \rm{dist}(x,\Pi^c)$ and 
observe that $B_{d(x)}(x)\subset \Pi$ and therefore
\[
\begin{split}
I(x,t)\leq \int_{B^c_{d(x)}(x)}P^s(x-y,t)\, dy &\leq n\omega_n t \int_{d(x)}^{\infty}\frac{\rho^{n-1}}{(\rho^2+t^\frac1s)^\frac{n+2s}{2}}\, d\rho
\\
&\leq
\frac{n\omega_n}{2s}\frac{t}{(d(x)^2+t^\frac1s)^s}.
\end{split}
\]
Integrating the above inequality on the set $(\pa\Pi)_\delta^-\setminus\bigcup_{i=1}^{N_\delta}Q_i^\delta$
gives
\begin{equation}\label{eq:case2}
\begin{split}
\int_{(\pa\Pi)_\delta^-\setminus\bigcup_{i=1}^{N_\delta}Q_\delta(x_i)}I(x,t)  \, dx&\leq \frac{n\omega_{n}}{s}\sum_{j=1}^J\int_{(\pa\Pi)_{\delta,j}^-\setminus\bigcup_{i=1}^{N_\delta}Q_\delta(x_i)}\frac{t}{(d(x)^2 +t^{\frac1s})^s} \, dx\\
&\leq \frac{n\omega_{n}}{s}\sum_{j=1}^J\int_{(\pa\Pi)_{\delta,j}^-\setminus\bigcup_{i=1}^{N_\delta}Q_\delta(x_i)}\frac{t}{[{\rm dist}(x,\Sigma_{\delta,j})]^s} \, dx\\
&\leq\frac{n\omega_{n} t}{s}\sum_{j=1}^J\int_{(\Sigma_{\delta,j})\setminus\bigcup_{i=1}^{N_\delta}Q_\delta(x_i)}\bigg(\int_0^\delta \frac{dr}{(r^2+t^\frac1s)^s}\bigg)\, d\Ha^{n-1}\\
&=
\frac{n\omega_{n} t^{\frac{1}{2s}}}{s}\sum_{j=1}^J\int_{(\Sigma_{\delta,j})\setminus\bigcup_{i=1}^{N_\delta}Q_\delta(x_i)}\bigg(\int_0^{\delta t^{-\frac{1}{2s}}} \frac{dr}{(r^2+1)^s}\bigg)\, d\Ha^{n-1}.
\end{split}
\end{equation}
Now we need to distinguish between the case $s=\frac12$ and $s>\frac12$. In the former case
we have
\[
\int_{0}^{\delta t^{-1}} \frac {1}{r^2+1}\, dr= \log(\delta+\sqrt{\delta^2+t^2})- \log(t)\leq |\log t|
\]
and therefore \eqref{eq:case2} becomes
\[
\int_{(\pa\Pi)_\delta^-\setminus\bigcup_{i=1}^{N_\delta}Q_\delta(x_i)}I(x,t)  \, dx
\leq 2n\omega_n t|\log t| \Ha^{n-1}\left(\bigg(\bigcup_{j=1}^J\Sigma_{\delta,j}\bigg)\setminus\bigcup_{i=1}^{N_\delta}Q_\delta(x_i)\right),
\]
While $s>\frac12$ from \eqref{eq:case2} we get
\[
\begin{split}
\int_{(\pa\Pi)_\delta^-\setminus\bigcup_{i=1}^{N_\delta}Q_i^\delta}I(x,t)  \, dx
\leq \frac{n\omega_{n} C  t^{\frac{1}{2s}} }{s}\Ha^{n-1}\left(\bigg(\bigcup_{j=1}^J\Sigma_{\delta,j}\bigg)\setminus\bigcup_{i=1}^{N_\delta}Q_i^\delta\right).
\end{split}
\]
with 
$$
C=\int_0^\infty \frac{1}{(r^2+1)^s}\, dr.
$$
Therefore, in any case we found that for $s\geq \frac12$ it holds
\beq\label{eq:case3}
\frac{1}{g_s(t)}\int_{(\pa\Pi)_\delta^-\setminus\bigcup_{i=1}^{N_\delta}Q_\delta(x_i)}I(x,t)
\leq C^2_{n,s} \Ha^{n-1}\left(\bigg(\bigcup_{j=1}^J\Sigma_{\delta,j}\bigg)\setminus\bigcup_{i=1}^{N_\delta}Q_\delta(x_i)\right)
\leq C^2_{n,s}\e.
\eeq
{\it Case three: }$x\in \Pi^- \cap \bigcup Q_{\delta}(x_i) $.
\\
 In this case we write
\[
I(x,t)=\int_{\Pi^c\cap \{y:|x-y|\geq \delta^2\}}P^s(x-y,t)\, dy +\int_{\Pi^c\cap \{y:|x-y|< \delta^2\}}P^s(x-y,t) \, dy.
\]
For the first integral, arguing  as in \eqref{eq:case1}, we quickly have
\[
\int_{\Pi^c\cap \{y:|x-y|\geq \delta^2\}}P^s(x-y,t)\, dy 
\leq \frac{\delta^2 t}{2s( \delta^4+t^2)^s}.
\]
Therefore, we now estimate that for each $1\leq i\leq N_\delta$ it holds
\beq\label{eq:case3.2}
\begin{split}
\int_{Q_{\delta}(x_i) \cap \Pi}\int_{\Pi^c\cap \{y:|x-y|< \delta^2\}}P^s(x-y,t)& \, dy
=
\int_{Q_{\delta}(x_i) \cap \Pi}\int_{Q_{\delta+\delta^2}(x_i)\cap\Pi^c} P^s(x-y,t) \, dy
\\
&\leq
\int_{Q_{\delta+\delta^2}(x_i) \cap \Pi}\int_{Q_{\delta+\delta^2}(x_i)  \cap \Pi^c} P^s(x-y,t) \, dy
\\
&\leq  \Gamma^{n,s}  \left(g_s(t)+ o(g_s(t))\right)\,   |Q_{\delta+\delta^2}|.
\end{split}
\eeq
Summing \eqref{eq:case3.2} for $1\leq i\leq N_\delta$ we get
\beq\label{eq:case3.1}
\frac{1}{g_s(t)}\int_{\bigcup_i Q_{\delta}(x_i)}\int_{\Pi^c} P^s(x-y,t)\, dxdy
\leq \Gamma^{n,s} N_{\delta}(\delta+\delta^2)^{n-1}
\leq \Gamma^{n,s} \Ha^{n-1}(\pa \Pi)+\e).
\eeq
Collecting \eqref{eq:case1},\eqref{eq:case2},\eqref{eq:case3.1},\eqref{eq:case3.2} and sending $t\to 0$ fisrt and $\e\to 0$ secondly we get
\[
\limsup_{t\to 0}\frac{1}{g_s(t)}\int_{\Pi}\int_{\Pi^c} P^s(x-y,t)\, dxdy \leq \Gamma^{n,s}{\rm P}(\Pi).
\]
The argument for a generic set $E$ of finite perimeter follows by density.
\end{proof}
\subsection{Characterization of sets of finite perimeter.}
As a consequence of our $\Gamma$-convergence analysis, we prove that a set $ E \in \mathrm{M}(\R^n)$ has finite perimeter if and only if for all $ s \in [\frac{1}{2},1)$
$$ \limsup_{t\to 0} \frac{\mathcal{E}^{s}_t(E)}{g^s(t)} < +\infty.$$
\begin{theorem}
	Let $E \in \mathrm{M}(\R^n)$. The following statement hold true.
	\begin{itemize}
		\item[(i)] If $ \limsup_{t\to 0} \frac{\mathcal{E}^{s}_t(E)}{g_s(t)} < +\infty$ for some $ s\in [\frac{1}{2},1)$, then $E$ is a set of finite perimeter.
		\item[(ii)] If $E$ is a set of finite perimeter then
		\begin{equation}\label{asd}
			\limsup_{t\to 0} \frac{\mathcal{E}^{s}_t(E)}{g_s(t)} < +\infty \quad \text{for every } s \in \big[\frac{1}{2},1\big).
		\end{equation}
	\end{itemize}
\end{theorem}
\begin{proof}
	We notice that (i) is an immediate consequence of \ref{Compthm} taking $E_i= E$ for every $i \in \N$. We prove (ii). By the very definition of $ \mathcal{E}^{s}_t(E)$ we have
	\begin{equation}\label{formpomfine}
		\begin{split}
			\frac{\mathcal{E}^{s}_t(E)}{g_s(t)}&= \frac{1}{g_s(t)} \int_{E} \int_{E^c}P^s(h,t) dh dx\\
			&= \frac{1}{2 g_s(t)} \int_{\R^n} \int_{\R^n} \vert \chi_{E}(x+h)-\chi_{E}(x) \vert P^s(h,t)dh dx \\
			& \leq \frac{\mathrm{P}(E)}{2 g_s(t)} \int_{\R^n} \vert h \vert P^s(h,t) dh.
		\end{split}
	\end{equation}
Let $ s \in (\frac{1}{2}, 1)$ by above formula and by \eqref{eq:decay} we obtain
\begin{equation}\label{formasdasd1}
	\begin{split}
	\limsup_{t\to 0^+}\frac{\mathcal{E}^{s}_t(E)}{g_s(t)} &\leq \limsup_{t\to 0^+} \frac{\mathrm{P}(E)}{2}\bigg[ \frac{1}{g_s(t)} \int_{B_{t^{\frac{1}{2s}}}} \vert h \vert t^{-\frac{n}{2s}} dh + \frac{1}{g_s(t)}\int_{B_{t^{\frac{1}{2s}}}^c} \vert h \vert \frac{t}{\vert h \vert^{n+2s}} dh  \bigg] \\
	&= \frac{\mathrm{P}(E)}{2}\big[ \frac{n \omega_{n}}{n-1}+ \frac{n \omega_{n}}{2s -1}  \big] 
	\end{split}
\end{equation}
hence we have that \eqref{asd} is true for all $ s \in (\frac{1}{2},1)$. For $s = \frac{1}{2}$ using the very definition of $P^{\frac{1}{2}}$, see \eqref{eng12}, we have
	\begin{equation}\label{formpomfine1}
	\begin{split}
	\limsup_{t\to 0^+} &	\frac{\mathcal{E}^{\frac{1}{2}}_t(E)}{g_s(t)}= \frac{1}{g_{\frac{1}{2}}(t)} \int_{E} \int_{E^c}P^{\frac{1}{2}}(x-y,t) dy dx\\
		&= 	\limsup_{t\to 0^+}\frac{1}{2 g_{\frac{1}{2}}(t)} \int_{\R^n} \int_{\R^n} \vert \chi_{E}(x+h)-\chi_{E}(x) \vert P^{\frac{1}{2}}(h,t)dh dx \\
		& \leq 	\limsup_{t\to 0^+}\frac{\mathrm{P}(E)}{2 t \vert \log (t) \vert} \bigg[ \int_{B_1} \frac{t \vert h \vert}{(\vert h\vert^2+t^2)^{\frac{n+1}{2}}} dh+  \int_{B_1^c} \frac{t \vert h \vert}{(\vert h\vert^2+t^2)^{\frac{n+1}{2}}} dh  \bigg] \\
		& = 	\limsup_{t\to 0^+} \frac{\mathrm{P}(E)}{2  \vert \log (t) \vert} \bigg[ \int_{B_{\frac{1}{t}}} \frac{\vert \eta \vert }{(\vert \eta \vert^2 +1)^{\frac{n+1}{2}}} d \eta+ \int_{B_{\frac{1}{t}}^c} \frac{\vert \eta \vert }{(\vert \eta \vert^2 +1)^{\frac{n+1}{2}}} d \eta  \bigg]  \\
		& = 	\limsup_{t\to 0^+} \frac{n \omega_{n}\mathrm{P}(E)}{2  \vert \log (t) \vert} \bigg[  \int_{0}^{\frac{1}{t}} \frac{\rho^n}{(\rho^2+1)^{\frac{n+1}{2}}}  d \rho+ \int_{\frac{1}{t}}^{+ \infty} \frac{\rho^n}{(\rho^2+1)^{\frac{n+1}{2}}} d \rho \bigg] = C(n)\mathrm{P}(E)
	\end{split}
\end{equation}
where in the last step we have used the L'Hôpital's rule. Hence by \eqref{formpomfine1} we obtain \eqref{asd} for $s = \frac{1}{2}$.
\end{proof}
\section{Application: The local and the non local isoperimetric inequality}
As a consequence of the $\Gamma$-convergence result we have a simple proof of the isoperimetric inequality.
Recall the fundamental properties of $P^s(z,t)$ that $\int_{\R^n}P^s(z,t)\, dz=1$ and $P^s(|z|,t)$ is a decreasing function of the modulus of $z$. 
Given a measurable set $E$ of finite measure and any $s\in (0,1)$ it is a straight forward application of Riesz rearrangement 
inequality
 that
\[\begin{split}
\int_{E}\int_{E^c}P^s(x-y,t)\,dxdy&=
|E|-\int_{E}\int_{E}P^s(x-y,t)\,dxdy\\&
\geq |E^*| -\int_{E^*}\int_{E^*} P^s(x-y)\, dxdy
\\
&= 
\int_{E^*}\int_{(E^*)^c}P^s(x-y,t)\,dxdy
\end{split}
\]
where $E^*$ denotes the ball centered at the origin such that $|E^*|=|E|$.
Therefore if $\Pi\subset \R^n$ is a 
polyhedron, 
we have
\[
\frac{1}{g_s(t)}\left(\int_{\Pi}\int_{\Pi^c}P^s(x-y,t)\,dxdy-
\int_{B}\int_{B^c}P^s(x-y,t)\, dxdy\right)\geq 0.
\]
with $|B|=|\Pi|$,
Recalling \eqref{eq:polyhedron}, \eqref{eq:liminf3} and the fact $\Gamma^{n,s}=\Gamma_{n,s}$, we take the $\limsup$ as $t\to 0$ in the above inequality
for $s \in [\frac12,1)$
to infer
\[
{\rm P}(\Pi)\geq{\rm P}(B)= n\omega_n^\frac{1}{n} |\Pi|^\frac{n-1}{n}
\]
and by density we recover the isoperimetric inequality for all sets of finite perimeter.
For $ s \in(0,\frac12)$ we even have pointwise convergence, hence it is immediate to get that any set $E$ with finite $2s-$perimeter it holds
\[
{\rm P}_{2s}(E)-{\rm P}_{2s}(B)\geq 0.
\]
which is the fractional isoperimetric inequality.
\section*{Data availability and conflict of interest}
On behalf of all authors, the corresponding author states that there is no conflict of interest.
No data was used for the research described in the article.
\section*{Aknowlegment}
Andrea Kubin was supported by the DFG Collaborative Research Center TRR 109 “Discretization in Geometry and Dynamics" and by the Academy of Finland grant 314227.
D.A. La Manna is a member of GNAMPA and has been partially supported by PRIN2022E9CF89. We thank L. Gennaioli and G. Stefani for spotting a mistake in the first draft of this paper.


\end{document}